\documentclass[10pt]{amsart}
\usepackage[style=alphabetic]{biblatex} 
\usepackage{mathrsfs}
\usepackage{color}
\usepackage{tikzit}
\usepackage{amssymb}
\usepackage{graphicx}
\usepackage[cmtip,all]{xy}
\usepackage{bbm}
\usepackage{comment}

\newtheorem{theorem}{Theorem}[section]
\newtheorem{lemma}[theorem]{Lemma}
\newtheorem{proposition}[theorem]{Proposition}
\newtheorem{corollary}[theorem]{Corollary}

\theoremstyle{definition}
\newtheorem{definition}[theorem]{Definition}
\newtheorem{example}[theorem]{Example}

\theoremstyle{remark}
\newtheorem{remark}[theorem]{Remark}
\numberwithin{equation}{section}

\renewcommand{\labelenumi}{(\roman{enumi})}

\DeclareMathOperator{\Hom}{Hom}

\DeclareMathOperator{\Ker}{Ker}
\DeclareMathOperator{\Coker}{Coker}
\let\Im\relax
\DeclareMathOperator{\Im}{Im}

\DeclareMathOperator{\Gal}{Gal}
\DeclareMathOperator{\Frac}{Frac}
\DeclareMathOperator{\dlog}{dlog}
\DeclareMathOperator{\height}{ht}

\def\category#1#2{\newcommand{#1}{{\mathbf {#2}}}}
\newcommand{\Set}{{\mathbf {Set}}}

\newcommand{\Sh}{\mathbf{Sh}}

\newcommand{\Abb}{\mathbb{A}}

\newcommand{\Dbb}{\mathbb{D}}

\newcommand{\Gbb}{\mathbb{G}}

\newcommand{\Pbb}{\mathbb{P}}

\newcommand{\Wbb}{\mathbb{W}}

\newcommand{\Zbb}{\mathbb{Z}}

\newcommand{\Bcal}{\mathcal{B}}

\newcommand{\Dcal}{\mathcal{D}}
\newcommand{\Ecal}{\mathcal{E}}
\newcommand{\Fcal}{\mathcal{F}}

\newcommand{\Ical}{\mathcal{I}}

\newcommand{\Ocal}{\mathcal{O}}
\newcommand{\Pcal}{\mathcal{P}}

\newcommand{\Xcal}{\mathcal{X}}
\newcommand{\Ycal}{\mathcal{Y}}
\newcommand{\Zcal}{\mathcal{Z}}

\def\symbol#1#2{\newcommand{#1}{{\mathrm {#2}}}}
\symbol{\del}{\partial}
\newcommand{\pr}{\mathrm{pr}}
\newcommand{\id}{\mathrm{id}}

\renewcommand{\tilde}{\widetilde}

\renewcommand{\bar}{\overline}

\DeclareMathOperator{\Spec}{Spec}

\DeclareMathOperator{\Pic}{Pic}
\DeclareMathOperator{\Res}{Res}

\newcommand{\Sm}{{\mathbf {Sm}}}
\newcommand{\Cor}{{\mathbf {Cor}}}
\newcommand{\PST}{{\mathbf {PST}}}
\newcommand{\NST}{{\mathbf {NST}}}

\newcommand{\RSC}{{\mathbf {RSC}}}
\newcommand{\MCor}{{\mathbf {MCor}}}
\newcommand{\MPST}{{\mathbf {MPST}}}

\newcommand{\bMCor}{{\mathbf {\underline{M}Cor}}}
\newcommand{\bMPST}{{\mathbf {\underline{M}PST}}}
\newcommand{\bMNST}{{\mathbf {\underline{M}NST}}}
\newcommand{\CI}{{\mathbf {CI}}}

\newcommand{\Nis}{\mathrm{Nis}}

\newcommand{\et}{\mathrm{\acute{e}t}}
\newcommand{\Ztr}{\mathbb{Z}_{\mathrm{tr}}}
\newcommand{\cube}{{\bar{\square}}}

\tikzstyle{new edge style 0}=[->]
\tikzstyle{new edge style 1}=[<-]
\tikzstyle{dashedline}=[-, dashed]

\calclayout

\symbol{\eff}{eff}
\symbol{\conn}{conn}
\symbol{\Fil}{Fil}
\symbol{\Art}{Art}
\symbol{\sm}{sm}
\symbol{\prop}{prop}
\symbol{\hyp}{hyp}
\symbol{\ho}{ho}

\symbol{\ur}{ur}
\symbol{\red}{red}
\symbol{\tr}{tr}
\symbol{\Nm}{Nm}
\symbol{\Tr}{Tr}
\symbol{\ch}{ch}
\symbol{\pro}{pro}
\symbol{\RGamma}{\mathrm{R}\Gamma}
\symbol{\Bl}{{Bl}}
\category{\Data}{Data}
\category{\SH}{SH}
\category{\D}{D}
\category{\Field}{Field}
\category{\bMSm}{\underline{M}Sm}
\category{\MSmfin}{MSm^{\mathrm{fin}}}
\category{\bMSmfin}{{\underline{M}Sm}^{\mathrm{fin}}}

\newcommand{\CIrec}{\mathbf{CI}^{\tau,sp}_{\Nis}}

\renewcommand{\div}{\mathrm{div}}

\begin{document}

\title{Steinberg symbols and reciprocity sheaves}
\author{Junnosuke Koizumi}
\date{}
\address{Graduate School of Mathematical Sciences, University of Tokyo, 3-8-1 Komaba, Meguro-ku, Tokyo 153-8914, Japan}
\email{jkoizumi@ms.u-tokyo.ac.jp}
\maketitle
\begin{abstract}
We study multilinear symbols on fields taking values in reciprocity sheaves.
We prove that any such symbol satisfying natural axioms automatically has Steinberg-type relations, which is a manifestation of the geometry of modulus pairs lying behind.
\end{abstract}
\setcounter{tocdepth}{1}

\tableofcontents

\section*{Introduction}

Suppose that $F(K)$ is some abelian group associated to a field $K$, such as the absolute differential forms or the Brauer group.
Let $r,s\geq 0$.
A multilinear map of the form
$$
	\varphi_K\colon K^r\times (K^\times)^s\to F(K)
$$
appearing in arithmetic geometry is often called a {\it symbol}.
Examples include the norm residue symbol $(K^\times)^s \to H^s_\et(\Spec K,\mu_\ell^{\otimes s})$ which is a generalization of the Hilbert symbol, its mod-$p$ variant $K\times (K^\times)^s \to H^1_\et(\Spec K,\Omega^s_{{-}/\mathbb{Z},\log})$, the differential symbol $(K^\times )^s \to \Omega^s_{K/\Zbb,\log}$ and its extension $K\times (K^\times)^s\to \Omega^s_{K/\Zbb}$ (see 2.2 for the definitions).
These symbols share the following properties in common:
	\begin{description}
		\item[(ST1)] $\varphi_K(\dots|\dots, a,\dots,1-a,\dots)=0.$
		\item[(ST2)] $\varphi_K(\dots,ca,\dots|\dots,a,\dots)+\varphi_K(\dots,c(1-a),\dots|\dots, 1-a, \dots)=0.$
	\end{description}	
The relation (ST1) is known as {\it Steinberg relation}.
The purpose of this paper is to reveal an algebro-geometric mechanism making these relations hold.

To formulate the problem, we use the notion of {\it reciprocity sheaf} defined in \cite{KSY}.
Let us fix a perfect base field $k$ and write $\Sm$ for the category of smooth $k$-schemes.
Roughly speaking, a reciprocity sheaf is a Nisnevich sheaf on $\Sm$ with transfers having a good ``ramification filtration''.
Since $k$ is a perfect field, any field $K$ over $k$ can be written as a union of smooth $k$-subalgebras $(A_i)_i$.
For a reciprocity sheaf $F$, we define $F(K)=\varinjlim_{i}F(A_i)$.
We define an {\it $(r,s)$-symbol} $\varphi$ for $F$ to be a family of multilinear maps
$$
	\varphi_K \colon K^r\times (K^\times)^s\to F(K)
$$
natural in $K$ and satisfying the projection formula (see Definition \ref{axiom} for the precise definition).
It is called {\it bounded} if the ``degree of ramification'' of $\varphi_K(a_1,\dots,a_r|b_1,\dots,b_s)$ is bounded from above by the sum of those of $a_1,\dots,a_r,b_1,\dots,b_s$.
Then the main result of this paper can be stated as follows.

\begin{theorem}[see Theorem \ref{main}]
	Let $F$ be a reciprocity sheaf.
	Let $\varphi$ be a bounded $(r,s)$-symbol for $F$.
	Then it satisfies (ST1).
	If $\ch(k)\neq 2$, then it also satisfies (ST2) and
	\begin{align*}\text{(ST3)}\quad&\varphi_K(\dots,bc,\dots,ad,\dots|\dots)+\varphi_K(\dots,ac,\dots,bd,\dots|\dots)\\
			{}={} &\varphi_K(\dots,c,\dots,abd,\dots|\dots) + \varphi_K(\dots,abc,\dots,d,\dots|\dots).\end{align*}
\end{theorem}

Our proof consists of two steps.
First we prove that an $(r,s)$-symbol for $F$ uniquely extends to a morphism from a certain Suslin homology sheaf to $F$.
Then we study this Suslin homology sheaf to deduce the Steinberg-type relations.
The key for the first step is the theory of unramified sheaves developed by Morel \cite{Mor} and the purity theorem for reciprocity sheaves due to Saito \cite{Sai20}.
We also prove in this step that an $(r,s)$-symbol is bounded if and only if it satisfies the Weil reciprocity; speaking slogan, we have
$$
	\text{bounded ramification} \iff \text{Weil reciprocity}.
$$

For the second step, we use the computation of the relevant Suslin homology sheaves in \cite{RSY}.
This suffices for (ST1) and (ST3), but not for (ST2) since their result is only available in the case $\ch(k)\not\in \{2,3,5\}$.
In this paper we improve their method to include the case $\ch(k)\in \{3,5\}$.
Our idea for proving (ST2) is to use the blow-up technique of Hu-Kriz \cite{HK01} which was originally used to prove (ST1) in the stable $\Abb^1$-homotopy category.
To do this, we need to study the behavior of a reciprocity sheaf (or, more generally, a semipure cube-invariant sheaf with $M$-reciprocity (see \ref{def:CIrec})) under a blow-up.
This can be done employing a result called ``logarithmic blow-up invariance'' which was recently proved by Saito \cite{Sai}.

We close this introduction with a brief discussion on related works.
Let $G_1,\dots,G_r$ be $\mathbb{A}^1$-invariant sheaves over $k$ and $F$ be a reciprocity sheaf over $k$.
Then our result (Theorem \ref{bounded_Weil}) implies that symbols of the form
$$
	G_1(K)\times\dots\times G_r(K)\to F(K)
$$
correspond to morphisms of sheaves $h_0^{\mathbb{A}^1,\Nis}(\bigotimes_i G_i)\to F$, where $h_0^{\mathbb{A}^1,\Nis}$ denotes the Suslin homology sheaf.
Kahn-Yamazaki \cite{KY} showed that the source of this morphism is isomorphic to the sheaf of Somekawa's $K$-groups $K({-},G_1,\dots,G_r)$, which is defined in \cite{Somekawa} in the case where $G_1,\dots,G_r$ are semiabelian varieties and \cite{KY} in general.
Therefore, our result provides an interpretation of Somekawa's $K$-groups in terms of symbols valued in reciprocity sheaves.

\subsection*{Acknowledgements}
The author is grateful to Shuji Saito for a lot of helpful comments, especially about the logarithmic blow-up invariance he proved.
The author would also like to thank Hiroyasu Miyazaki for several inspiring discussions on the cohomology of reciprocity sheaves.
The author thanks anonymous referees for their valuable comments.

\section{Preliminaries}
We fix once and for all a perfect base field $k$.
An {\it algebraic scheme over $k$} is a separated scheme of finite type over $k$.
A {\it smooth scheme over $k$} is an algebraic scheme over $k$ which is smooth over $k$.
Let $\Sm$ denote the category of smooth schemes over $k$.
For these terms, we will omit ``over $k$'' hereinafter.
Unless otherwise specified, the term ``(pre)sheaf'' will mean ``(pre)sheaf of abelian groups''.

\subsection{Big Witt vectors}
First we recall the definition of the group of big Witt vectors.
Let $A$ be a ring and $n\geq 0$.
The {\it group of big Witt vectors of length $n$} is defined as
$$
	\Wbb_n(A):=1+TA[T]/(T^{n+1})\subset (A[T]/(T^{n+1}))^\times.
$$
For example, $\Wbb_0(A)=0$ and $\Wbb_1(A)\simeq A;\,1-aT\mapsto a$.
It is easily seen that any element $f$ of $\Wbb_n(A)$ can be uniquely written in the form $f = \prod_{i=1}^n (1-a_iT^i)$.
We write $f=(a_1,\dots,a_n)$ for this element.
This expression gives a bijection $\Wbb_n(A)\xrightarrow{\sim} A^n$ which is not a group homomorphism in general.
For $a\in A$, we set
$$
	[a] := 1-aT \in \mathbb{W}_n(A)
$$
and call it the {\it Teichm\"uller representative}.
We define $\Wbb_n(X):=\Wbb_n(\Ocal_X(X))$ for $X\in \Sm$.
The presheaf $\Wbb_n$ on $\Sm$ is representable by a group scheme whose underlying scheme is $\Abb^n$.

We also define an abelian group $\mathbb{D}_n(A)$ by $\mathbb{D}_n(A):=\mathbb{W}_n(A)\oplus \mathbb{W}_n(A)\oplus \mathbb{Z}\oplus A^\times$ for each $n\geq 0$.
For $a\in A^\times$, we set
$$
	[\![a]\!] := (1-a^{-1}T, 1-aT, 1, a).
$$
We define $\Dbb_n(X):=\Dbb_n(\Ocal_X(X))$ for $X\in \Sm$.
The presheaf $\Dbb_n$ on $\Sm$ is representable by a group scheme whose underlying scheme is $\Abb^n\times \Abb^n\times \Zbb\times (\Abb^1\setminus\{0\})$.

An important aspect of the group of big Witt vectors is that it is related to relative Picard groups.
Consider $X\in \Sm$ and an effective Cartier divisor $D$ on $X$.
Then the relative Picard group $\Pic(X,D)$ is defined as
$$
	\Pic(X,D):=\{(L,s)\mid L\colon\text{line bundle on }X,\,s\colon \Ocal_D\xrightarrow{\sim}L|_D\}/\simeq.
$$
If $Z$ is an effective Cartier divisor on $X$ disjoint from $D$, then we have an associated element $[Z]:=(\Ocal(Z),\id)\in \Pic(X,D)$, where $\Ocal(Z)$ denotes the sheaf of functions with poles in $Z$.

\begin{proposition}\label{relative_picard}
	Let $X$ be an affine smooth scheme.
	Let $T$ denote the coordinate of $\Abb^1$.
	\begin{enumerate}
		\item There is an isomorphism $\Pic(\mathbb{P}^1_X,(n+1)\{\infty\})/\Zbb[0] \simeq \Wbb_n(X)$ which is given by
			\begin{align}\label{rel_pic_fml_1}
				[\div(a_dT^d+\dots+a_0)]\mapsto \sum_{i=0}^d (a_{d-i}/a_d)T^i.
			\end{align}
		\item There is an isomorphism $\Pic(\mathbb{P}^1_X,\{0,\infty\})/\Zbb[1] \simeq \Gbb_m(X)$ which is given by
			\begin{align}\label{rel_pic_fml_2}
				[\div(a_dT^d+\dots+a_0)]\mapsto (-1)^da_0/a_d.
			\end{align}
		\item There is an isomorphism $\Pic(\mathbb{P}^1_X,(n+1)\{0,\infty\})\simeq \Dbb_n(X)$ which is given by
			\begin{align}\label{rel_pic_fml_3}
				[\div(a_dT^d+\dots+a_0)]\mapsto \biggl(\sum_{i=0}^d (a_i/a_0)T^i, \sum_{i=0}^d (a_{d-i}/a_d)T^i, d, (-1)^da_0/a_d\biggr).
			\end{align}
	\end{enumerate}
\end{proposition}

\begin{proof}
	(i) Let $D=(n+1)\{\infty\}$ and $A=\Ocal_X(X)$.
	We have an isomorphism $\Pic(\mathbb{P}^1_X)\simeq \Pic(X)\oplus \mathbb{Z}$.
	For each element $(L,s)\in \Pic(\mathbb{P}^1_X,D)$, $s$ gives a trivialization of $L|_{\{\infty\}}$ and hence $L\simeq \Ocal(d)$ for some $d\in \mathbb{Z}$.
	Therefore $\Pic(\mathbb{P}^1_X,D)/\Zbb[0]$ is generated by the elements of the form $(\Ocal_{\mathbb{P}^1_X},s)$.
	A trivialization $s\colon \Ocal_D\xrightarrow{\sim} \Ocal_{\mathbb{P}^1_X}|_D$ is given by an element of $\Ocal_D(D)^\times$, i.e. an invertible truncated polynomial $f\in A[T]/(T^{n+1})^\times$.
	Two polynomials $f$ and $f'$ defines a same element in $\Pic(\mathbb{P}^1,D)/\Zbb[0]$ if and only if there is some $a\in A^\times$ such that $f=af'$.
	Hence we get $\Pic(\mathbb{P}^1,D)/\Zbb[0]\simeq (A[T]/(T^{n+1}))/A^\times\simeq \mathbb{W}_n(A)$.
	(ii) and (iii) can be proved similarly; see \cite[Lemma 2.2 and Remark 2.3]{BRS} for a detailed proof of (iii).
\end{proof}

\begin{remark}
	The isomorphism given in Proposition \ref{relative_picard} (i) sends $[\div(T-a)]$ to $[a]$.
	In other words, the rational point $a\in \mathbb{A}^1$ corresponds to the Teichm\"uller representative $[a]$.
	Similarly, the isomorphism given in Proposition \ref{relative_picard} (iii) sends $[\div(T-a)]$ to $[\![a]\!]$.
\end{remark}

\subsection{Unramified sheaves}
We recall from \cite{Mor} the notion of unramified sheaf.

We define $\Sm^\pro$ to be the category of noetherian regular schemes over $k$ of the form $X=\varprojlim_{i}X_i$, where $(X_i)_i$ is a projective system of smooth $k$-schemes with affine transition maps.
Then we can regard $\Sm^\pro$ as a subcategory of the category of pro-objects of $\Sm$ (see \cite[Lemma 2.5]{RS21}).
Consider a presheaf $F$ of sets on $\Sm$.
For $X=\varprojlim_iX_i\in \Sm^\pro$, we can define $F(X):=\varinjlim_iF(X_i)$. 
For example, we have $F(\Ocal_{X,x})=\varinjlim_{x\in U} F(U)$ for $X\in \Sm$ and $x\in X$.
Let $\Field_k$ denote the category of fields over $k$.
Since $k$ is perfect, any $K\in \Field_k$ is a union of smooth $k$-subalgebras.
Therefore $F$ induces a colimit-preserving functor $\Field_k\to \Set$. 

\begin{definition}
	Let $F$ be a Nisnevich sheaf of sets on $\Sm$.
	\begin{enumerate}
		\item We say that $F$ has {\it global injectivity} if for any connected $X\in \Sm$, the restriction map $F(X)\to F(k(X))$ is injective.
		\item $F$ is called an {\it unramified sheaf} if it has global injectivity and $F(X)=\bigcap_{x\in X^{(1)}}F(\Ocal_{X,x})$ holds for any connected $X\in \Sm$.
	\end{enumerate}
\end{definition}

\begin{example}
	The sheaves $\Gbb_a\colon X\mapsto \Ocal(X)$ and $\Gbb_m\colon X\mapsto \Ocal(X)^\times$ are unramified.
	This is a consequence of the fact that $A=\bigcap_{\height\mathfrak{p}=1} A_{\mathfrak{p}}$ holds for a noetherian normal domain $A$.
\end{example}

\begin{theorem}[{\cite[Theorem 2.11]{Mor}}]\label{unramified_lemma_Morel}
	Let $F,G$ be unramified sheaves of sets and $\{\varphi_K\colon F(K)\to G(K)\}_{K\in \Field_k}$ be a family of maps natural in $K\in \Field_k$.
	Suppose that for any connected $X\in \Sm$ and $x\in X^{(1)}$, we have $\varphi_{k(X)}(F(\Ocal_{X,x}))\subset G(\Ocal_{X,x})$ and the diagram
	$$
	\xymatrix{
		F(\Ocal_{X,x})\ar[r]^-{\varphi_{\Ocal_{X,x}}}\ar[d]		&		G(\Ocal_{X,x})\ar[d]\\
		F(k(x))\ar[r]^-{\varphi_{k(x)}}										&		G(k(x))
	}
	$$
	is commutative, where $\varphi_{\Ocal_{X,x}}$ denotes the restriction of $\varphi_{k(X)}$.
	Then $\{\varphi_K\}_{K\in \Field_k}$ extends to a morphism $\varphi\colon F\to G$ of sheaves of sets.
\end{theorem}

\subsection{Sheaves with transfers}

We recall some terminology from Voevodsky's theory of sheaves with transfers; see \cite{MVW} for details.

Let $X,Y \in \Sm$.
An {\it elementary correspondence} from $X$ to $Y$ is an integral closed subscheme $V\subset X\times Y$ which is finite and surjective over some connected component of $X$.
A {\it finite correspondence} from $X$ to $Y$ is a $\Zbb$-linear combination of elementary correspondences.
Let $\Cor(X,Y)$ denote the abelian group of finite correspondences.
For example, the graph $\Gamma_f$ of a morphism $f\colon X\to Y$ gives a finite correspondence $[\Gamma_f] \in \Cor(X,Y)$, for which we simply write $f$.
If $f$ is finite and surjective, then the transpose of $\Gamma_f$ gives a finite correspondence $[{}^t\Gamma_f]\in \Cor(Y,X)$ in the other direction, for which we write ${}^tf$.
There is a bilinear composition map
$$
	\circ\colon \Cor(Y,Z) \times \Cor(X,Y)\to \Cor(X,Z)
$$
defined by intersection products.
We can form an additive category $\Cor$ of smooth schemes and finite correspondences.
Moreover, we can embed $\Sm$ into $\Cor$ by the graph construction.
The fiber product of schemes (over $k$) gives a symmetric monoidal structure $\times$ on $\Cor$.
Let $\Cor^\pro$ denote the category whose objects are the same as $\Sm^\pro$ and morphisms are finite correspondences.
Then we can regard $\Cor^\pro$ as a full subcategory of the category of pro-objects of $\Cor$ (see \cite[Lemma 2.5]{RS21}).

A {\it presheaf with transfers} is a presheaf on $\Cor$.
It is called a {\it Nisnevich sheaf} if it is a Nisnevich sheaf as a presheaf on $\Sm$.
Let $\PST$ be the category of presheaves with transfers and $\NST$ its full subcategory of Nisnevich sheaves.
The inclusion functor $\NST\to \PST$ admits a left adjoint $({-})_\Nis$ given by the ordinary Nisnevich sheafification.
In particular, $\NST$ is an abelian category and $({-})_\Nis$ is an exact functor.
For $F\in \PST$, a finite surjective morphism $f\colon X\to Y$ induces a homomorphism $f_*:=({}^tf)^*\colon F(X)\to F(Y)$.
In particular, a finite extension $L/K$ in $\Field_k$ induces a homomorphism $\Tr_{L/K}\colon F(L)\to F(K)$.

\begin{example}
	Let $G$ be a smooth commutative group scheme over $k$.
	Then $G$, viewed as a Nisnevich sheaf on $\Sm$, has a canonical structure of a Nisnevich sheaf with transfers; see \cite[Lemma 1.4.4]{BVK16}.
	For example, we can regard $\Gbb_a, \Gbb_m, \Wbb_n\,(n\geq 0)$ and $\Dbb_n\,(n\geq 0)$ as objects of $\NST$.
\end{example}

\begin{example}
	For $X\in \Cor$, let $\Ztr(X)$ denote the presheaf with transfers represented by $X$, i.e. $\Ztr(X)(Y)=\Cor(Y,X)$.
	Then $\Ztr(X)$ is a Nisnevich sheaf.
\end{example}

An object $F\in \PST$ is called {\it $\Abb^1$-invariant} if the canonical homomorphism
$\pi^*\colon F(X)\to F(X\times \Abb^1)$
is an isomorphism for any $X\in \Sm$, where $\pi\colon X\times \Abb^1\to X$ is the projection.
One of the main results of Voevodsky's theory is the following.

\begin{theorem}[{\cite[Theorem 22.3]{MVW}}]\label{a1_invariant_sheafification}
	If $F\in \PST$ is $\mathbb{A}^1$-invariant, then $F_\Nis$ is also $\mathbb{A}^1$-invariant.
\end{theorem}

\begin{lemma}\label{injectivity_factorization_lemma}
	Let $\varphi\colon F\to G$ and $\psi\colon F\to F'$ be morphisms in $\NST$ and suppose that $\psi$ is an epimorphism.
	If $G$ has global injectivity, then $\varphi$ factors through $F'$ if and only if $\varphi_K$ factors through $F'(K)$ for each $K\in \Field_k$.
\end{lemma}

\begin{proof}
	Suppose that $\varphi_K$ factors through $F'(K)$ for each $K\in \Field_k$.
	Let $F''=\Ker \psi$.
	We have to prove that $\varphi(F'')_\Nis\subset G$ is zero.
	Since $\varphi(F'')_\Nis$ is a subsheaf  of $G$, it has global injectivity.
	Since have $\varphi(F'')_\Nis(K)=0$ for each $K\in \Field_k$ by assumption, we get $\varphi(F'')_\Nis=0$.
\end{proof}

The symmetric monoidal structure $\times$ on $\Cor$ extends by colimits to a symmetric monoidal structure $\otimes^\PST$ on $\PST$ and $\otimes^\NST$ on $\NST$.
They satisfy $\Ztr(X)\otimes \Ztr(Y) \simeq \Ztr(X\times Y)$ by definition and the sheafification functor $({-})_\Nis\colon \PST\to \NST$ is symmetric monoidal.
Let $F_1,\dots,F_r\in \NST$, $X_1,\dots,X_r\in \Sm$ and $a_i\in F_i(X_i)$.
Then we have an induced morphism
$
	a_1\boxtimes \dots\boxtimes a_r \colon \Ztr(\prod_{i=1}^r X_i)\to \bigotimes_i^\NST F_i
$
called the {\it exterior product} of $a_1,\dots,a_r$.
If $X_1=\dots=X_r=X$, then pulling back by the diagonal morphism $\Delta_X\colon X\to X^r$ yields an element $a_1 \otimes \dots \otimes a_r=\Delta_X^*(a_1\boxtimes \dots\boxtimes a_r)\in (\bigotimes_i^\NST F_i)(X)$ called the {\it tensor product} of $a_1,\dots,a_r$.

\subsection{Modulus pairs}

The theory of modulus pairs provides a natural framework to deal with ramification-theoretic data.
We recall some terminology here; see \cite{KSY} for details.

A {\it modulus pair} $\Xcal=(X,D_X)$ consists of an algebraic scheme $X$ and an effective Cartier divisor $D_X$ on $X$ such that $X^\circ:=X\setminus |D_X|$ is smooth.
It is called {\it proper} if $X$ is proper.
A basic example is the cube $\cube := (\Pbb^1,\{\infty\})$.
An {\it ambient morphism} $\Xcal\to \Ycal$ of modulus pairs is a morphism $f\colon X^\circ \to Y^\circ$ which extends to a morphism $\bar{f}\colon X\to Y$ such that $\pi^*D_X\geq \pi^*\bar{f}^*D_Y$, where $\pi\colon X^N\to X$ is the normalization.
We define $\bMSmfin$ (resp. $\MSmfin$) to be the category of modulus pairs (resp. proper modulus pairs) and ambient morphisms.
For modulus pairs $\Xcal$ and $\Ycal$, their {\it box product} is defined by $\Xcal\boxtimes\Ycal = (X\times Y, \pr_1^*D_X+\pr_2^*D_Y)$.

Let $\Xcal,\Ycal$ be modulus pairs.
An elementary correspondence $V$ from $X^\circ$ to $Y^\circ$ is called {\it left proper} if its closure $\bar{V}\subset X\times Y$ is proper over $X$, and it is called {\it admissible} if $p^*D_X\geq q^*D_Y$, where $p\colon \bar{V}^N\to X$ and $q\colon \bar{V}^N\to Y$ are projections.
A finite correspondence $\alpha\in \Cor(X^\circ,Y^\circ)$ is called left proper (resp. admissible) if it is a $\Zbb$-linear combination of left proper (resp. admissible) elementary correspondences.
Let $\bMCor(\Xcal,\Ycal)$ denote the abelian group of left proper admissible finite correspondences.
For example, the graph $\Gamma_f$ of an ambient morphism $f\colon \Xcal\to \Ycal$ gives an element $[\Gamma_f]\in \bMCor(\Xcal,\Ycal)$.
The composition of finite correspondences restricts to
$$
	\circ\colon \bMCor(\Ycal,\Zcal) \times \bMCor(\Xcal,\Ycal)\to \bMCor(\Xcal,\Zcal)
$$
and we get an additive category $\bMCor$ of modulus pairs and left proper admissible finite correspondences.
Let $\MCor$ denote its full subcategory of proper modulus pairs.
The box product gives a symmetric monoidal structure $\boxtimes$ on $\bMCor$ and on $\MCor$.
There is a sequence of symmetric monoidal functors
\begin{align}\label{mcor_functors}
	\xymatrix{
		\MCor\ar[r]^-{\tau}\ar@/_12pt/[rr]_-{\omega}				&\bMCor\ar[r]^-{\underline{\omega}}			&\Cor
	}
\end{align}
where $\tau$ is the inclusion and $\underline{\omega}(\Xcal)=X^\circ$.
We define the category $\bMCor^\pro$ as follows:
\begin{itemize}
	\item The objects are pairs $\Xcal=(X,D_X)$ where 
	\begin{enumerate}
		\item $X$ is a separated noetherian scheme over $k$ of the form $X = \varprojlim_i X_i$ where $(X_i)_i$ is a projective system of algebraic schemes,
		\item $D_X=\varprojlim_i D_{X_i}$ where $D_{X_i}$ is an effective Cartier divisor on $X_i$ such that $X_i\setminus|D_{X_i}|$ is smooth, $D_{X_i}|_{X_j}=D_{X_j}$ for $i\leq j$ and $X\setminus |D_X|$ is regular.
	\end{enumerate}
	\item The morphisms are left proper admissible finite correspondences.
\end{itemize}
We can embed $\bMCor^\pro$ into the category of pro-objects of $\bMCor$ (see \cite[Lemma 3.8]{RS21}).

We let $\bMPST$ (resp. $\MPST$) be the category of presheaves on $\bMCor$ (resp. $\MCor$).
For any $G\in \bMPST$ and  a modulus pair $\mathcal{X}$, we can define a presheaf $G_\mathcal{X}$ on $X_\Nis$ by
$$
	(U \xrightarrow{\pi} X) \mapsto G(U, \pi^*D_X).
$$
A presheaf $G\in \bMPST$ is called a {\it Nisnevich sheaf} if $G_\mathcal{X}$ is a Nisnevich sheaf for any modulus pair $\mathcal{X}$.
Let $\bMNST$ denote the full subcategory of $\bMPST$ consisting of Nisnevich sheaves.
The inclusion functor $\bMNST\to \bMPST$ admits a left adjoint $({-})_\Nis$ (but it is not true in general that $(G_\Nis)_{\Xcal}\simeq (G_{\Xcal})_\Nis$).
In particular, $\bMNST$ is an abelian category and $({-})_\Nis$ is an exact functor.
We write $H^i_\Nis(\Xcal,G):=H^i_\Nis(X,G_\Xcal)$ for $G\in \bMNST$.

\begin{example}
	For $\Xcal\in \bMCor$ (resp. $\MCor$), let $\Ztr(\Xcal)$ denote the object of $\bMPST$ (resp. $\MPST$) represented by $\Xcal$, i.e. $\Ztr(\Xcal)(\Ycal)=\bMCor(\Ycal,\Xcal)$ (resp. $\MCor(\Ycal,\Xcal)$).
\end{example}

The functors (\ref{mcor_functors}) extends by colimits to a sequence of right exact functors
$$
	\xymatrix{
		\MPST\ar[r]^-{\tau_!}\ar@/_12pt/[rr]_-{\omega_!}				&\bMPST\ar[r]^-{\underline{\omega}_!}			&\PST.
	}
$$
These functors can be described explicitly as follows:
\begin{itemize}
	\item	For $G\in \MPST$, $\tau_!G$ is given by $\Xcal \mapsto \varinjlim_{\Ycal\in \mathbf{Comp}(\Xcal)}G(\Ycal)$.
	\item	For $G\in \bMPST$, $\underline{\omega}_!G$ is given by $X\mapsto G(X,\emptyset)$.
	\item	For $G\in \MPST$, $\omega_!G$ is given by $X \mapsto \varinjlim_{\Ycal\in \mathbf{Comp}(X,\emptyset)}G(\Ycal)$.
\end{itemize}
Here $\mathbf{Comp}(\Xcal)$ denotes the category of {\it compactifications} of $\Xcal$, i.e. those $\Ycal = (Y,D+\Sigma)\in \MCor$ with $Y\setminus|\Sigma|=X$ and $D|_X=D_X$.
Moreover, these functors admit a right adjoint:
\begin{itemize}
	\item	$\tau_!$ has a right adjoint $\tau^*$ which is given by $\tau^*G(\Xcal)=G(\Xcal)$.
	\item	$\underline{\omega}_!$ has a right adjoint $\underline{\omega}^*$ which is given by $\underline{\omega}^*F(\Xcal)=F(X^\circ)$.
	\item	$\omega_!$ has a right adjoint $\omega^*$ which is given by $\omega^*F(\Xcal)=F(X^\circ)$.
\end{itemize}
Note that $\tau_!$, $\underline{\omega}^*$ and $\omega^*$ are fully faithful since $\tau^*\tau_!\simeq \id$, $\underline{\omega}_!\underline{\omega}^*\simeq \omega_!\omega^*\simeq \id$.
We write $\tau_!^\Nis G := (\tau_!G)_\Nis$ and $\omega_!^\Nis G := (\omega_!G)_\Nis$.

\begin{lemma}[{\cite[Proposition 6.2.1]{KMSY21}}]\label{omega_sheafification}
	For any $G\in \bMPST$, the canonical morphism $(\underline{\omega}_!G)_\Nis \to \underline{\omega}_!(G_\Nis)$ is an isomorphism.
\end{lemma}

The symmetric monoidal structure $\boxtimes$ on $\bMCor$ (resp. $\MCor$) extends by colimits to a symmetric monoidal structure $\otimes^\bMPST$ (resp. $\otimes^\MPST$) on $\bMPST$ (resp. $\MPST$).
The functors $\tau_!, \underline{\omega}_!, \omega_!$ and $\omega_!^\Nis$ are symmetric monoidal.

Let $G_1,\dots,G_r\in \bMPST$, $\Xcal_1,\dots,\Xcal_r\in \bMCor$ and $a_i\in G_i(\Xcal_i)$.
Then we can define the {\it exterior product} $a_1\boxtimes\dots\boxtimes a_r\in (\bigotimes_i^\bMPST G_i)(\boxtimes_{i=1}^r\Xcal_i)$ as for $\otimes^\NST$.
If $X_1=\dots=X_r=X$, then pulling back by the diagonal morphism $\Delta_X\colon (X,\sum_{i=1}^rD_{X_i})\to \boxtimes_{i=1}^r\Xcal_i$ yields the {\it tensor product}
$$
	a_1\otimes\dots\otimes a_r=\Delta_X^*(a_1\boxtimes\dots\boxtimes a_r)\in \biggl(\bigotimes_i^\bMPST G_i\biggr)\biggl(X,\sum_{i=1}^rD_{X_i}\biggr).
$$

\subsection{Cube-invariant sheaves and reciprocity sheaves}
An object $G\in \MPST$ is called {\it cube-invariant} if the canonical homomorphism $\pi^*\colon G(\Xcal)\to G(\Xcal\boxtimes \cube)$ is an isomorphism for any $\Xcal\in \MCor$, where $\pi\colon \Xcal\boxtimes \cube\to \Xcal$ is the projection.
We write $\CI$ for the full subcategory of $\MPST$ consisting of cube-invariant objects.
The inclusion functor $\CI\to \MPST$ has a left adjoint $h^0_\cube$ and a right adjoint $h_0^\cube$ which are given by
\begin{align*}
	h_0^\cube G(\Xcal) &= \Coker (G(\Xcal\boxtimes \cube)\xrightarrow{i_0^*-i_1^*} G(\Xcal)),\\
	h^0_\cube G(\Xcal) &= \Hom_\MPST(h_0^\cube(\Xcal),G).
\end{align*}
By definition, $h_0^\cube G$ is the maximal cube-invariant quotient presheaf of $G$ and $h^0_\cube G$ is the maximal cube-invariant sub-presheaf of $G$.

\begin{lemma}
	For any $G_1,\dots,G_r\in \MPST$, the canonical morphism
	$h_0^\cube(\bigotimes_i^\MPST G_i)\to h_0^\cube(\bigotimes_i^\MPST h_0^\cube G_i)$
	is an isomorphism.
\end{lemma}

\begin{proof}
	This is a formal consequence of \cite[Proposition 2.1.9]{RSY}.
\end{proof}

For $G\in \MPST$ we define $h_0 G:=\omega_!h_0^\cube G$ and $h_0^\Nis G:= (h_0 G)_\Nis$.
For a proper modulus pair $\Xcal$, we write $h_0^\cube \Xcal$ for $h_0^\cube \Ztr(\Xcal)$ and similarly for $h_0$, $h_0^\Nis$.

\begin{lemma}\label{suslin_homology}
	Let $\Xcal\in \MCor$.
	Then for any $Y\in \Sm$ we have
	$$
		h_0\Xcal(Y) \simeq \Coker(\bMCor(Y\boxtimes \cube,\Xcal)\xrightarrow{i_0^*-i_1^*} \Cor(Y,X^\circ)).
	$$
\end{lemma}

\begin{proof}
	This follows from the definition of $h_0^\cube$ and the explicit description of $\omega_!$.
\end{proof}

The right hand side of Lemma \ref{suslin_homology} is the same thing as the Suslin homology with modulus $H_0^S((Y\times X)/Y, \pr_2^*D_X)$; see \cite[Definition 3.1]{RY16} for the definition.

We define $\RSC$ to be the essential image of $\omega_!\colon \CI\to \PST$ and $\RSC_\Nis:=\RSC\cap \NST$.
An object of $\RSC$ (resp. $\RSC_\Nis$) is called a {\it reciprocity presheaf} (resp. {\it reciprocity sheaf}).
$\RSC\subset \PST$ is closed under subobjects and quotient objects, and hence is an abelian subcategory of $\PST$.
By definition, $\omega_!\colon \CI\to \RSC$ has a right adjoint $h^0_\cube\omega^*$.
We define $\tilde{F}:=h^0_\cube\omega^*F$ for $F\in \RSC_\Nis$.
It is known that $F\mapsto \tilde{F} $ is fully faithful, i.e. $\omega_!\tilde{F}\simeq F$.
The next result is fundamental but highly non-trivial:

\begin{theorem}[{\cite[Theorem 0.1]{Sai20}}]
	$({-})_\Nis\colon \PST\to \NST$ sends $\RSC$ to $\RSC_\Nis$.
\end{theorem}

This theorem shows that $\RSC_\Nis\subset \NST$ is closed under subobjects and quotient objects, and hence is an abelian subcategory of $\NST$.
It also shows that $\omega_!^\Nis$ sends $\CI$ to $\RSC_\Nis$ and $\tilde{({-})}$ gives a right adjoint of it.
The following deep result is also due to Saito:

\begin{theorem}[Purity, {\cite[Theorem 0.2]{Sai20}}]\label{purity}
	Let $F\in \RSC_\Nis$, $X\in \Sm$ and $x\in X^{(n)}$.
	Then we have
	$H^i_{x,\Nis}(X,F)=0$ for $i\neq n$.
\end{theorem}

\begin{corollary}\label{reciprocity_unramified}
	Reciprocity sheaves are unramified.
\end{corollary}

\begin{proof}
	Let $F$ be a reciprocity sheaf.
	Fix a connected $X\in \Sm$ and consider the localization spectral sequence $E_1^{p,q}=\bigoplus_{x\in X^{(p)}}H^{p+q}_{x,\Nis}(\Spec \Ocal_{X,x},F)\Rightarrow H^{p+q}_\Nis(X,F)$.
	The $E_1$ page is concentrated in $q=0$ by Theorem \ref{purity} and thus we get an exact sequence
	$$
		0\to F(X)\to F(k(X))\xrightarrow{\del} \bigoplus_{x\in X^{(1)}}H^1_{x,\Nis}(\Spec \Ocal_{X,x},F).
	$$
	Since the kernel of $\del_x\colon F(k(X))\to H^1_{x,\Nis}(\Spec \Ocal_{X,x},F)$ is $F(\Ocal_{X,x})$, this proves that $F$ is unramified.
\end{proof}

\begin{example}
	Any smooth commutative group scheme $G$ over $k$, seen as an object of $\NST$, is a reciprocity sheaf (see \cite[Corollary 3.2.5]{KSY}).
	In particular, $\Gbb_a,\,\Gbb_m,\,\Wbb_n$ and $\Dbb_n$ can be regarded as reciprocity sheaves.
\end{example}

\begin{example}
	If $F\in \PST$ is $\Abb^1$-invariant, then it is a reciprocity presheaf since $\omega^*F\in \CI$ and $F\simeq \omega_!\omega^*F$.
	For example, the Milnor $K$-theory sheaf $K^M_n$ is a reciprocity sheaf.
\end{example}

\begin{example}
	The sheaf of absolute differential forms $\Omega^n_{{-}/\Zbb}$ is a reciprocity sheaf (see \cite[Corollary 3.2.5]{KSY}).
\end{example}

\begin{example}
	This example is taken from \cite[Section 5]{RSY}.
	For $X\in \Sm$, we define its {\it first infinitesimal neighborhood of the diagonal} $\Pcal^1(X)$ by
	$$
		\Pcal^1(X):=\Gamma(X\times_{\Spec \Zbb} X,\Ocal_{X\times_{\Spec \Zbb} X}/\Ical_\Delta^2)
	$$
	where $\Ical_\Delta$ is the defining ideal of the diagonal.
	If $X=\Spec A$, then $\Pcal^1(X)$ is the quotient of $A\otimes_\Zbb A$ by the relation
	\begin{align*}
		&(c\otimes d)(a\otimes 1-1\otimes a)(b\otimes 1-1\otimes b)\\
		=&(c\otimes abd+abc\otimes d)-(bc\otimes ad + ac\otimes bd)=0.
	\end{align*}
	There is an isomorphism of presheaves $\Pcal^1\xrightarrow{\sim} \Gbb_a\oplus \Omega^1_{{-}/\Zbb}$ given by $a\otimes b\mapsto (ab,adb)$.
	Hence $\Pcal^1$ has a structure of a reciprocity sheaf.
\end{example}

\begin{example}\label{etale_unramified}
	Let $\varepsilon\colon (\Sm)_{\et}\to (\Sm)_\Nis$ be the morphism of sites given by the identity functor.
	For a prime number $\ell$ invertible in $k$, $R^n\varepsilon_*(\mu_\ell^{\otimes m}) = (X\mapsto H^n_\et(X,\mu_\ell^{\otimes m}))_\Nis$ is a reciprocity sheaf since it is $\Abb^1$-invariant by Theorem \ref{a1_invariant_sheafification}.
	If $\ch(k)=p>0$, then $R^n\varepsilon_*\Omega^m_{{-}/\mathbb{Z},\log}=(X\mapsto H^n_\et(X,\Omega^m_{{-}/\mathbb{Z},\log}))_\Nis$ is also a reciprocity sheaf.
	Indeed, we have an exact sequence of \'etale sheaves
	\begin{align}\label{Cartier_sequence}
		0\to \Omega^m_{{-}/\mathbb{Z},\log}\to \Omega_{{-}/\Zbb}^m\xrightarrow{1-C^{-1}}\Omega_{{-}/\Zbb}^m/d\Omega_{{-}/\Zbb}^{m-1}\to 0
	\end{align}
	where $C^{-1}$ is the inverse Cartier operator.
	Since the middle term and the right term are $R\varepsilon_*$-acyclic, we get $R\varepsilon_*\Omega^m_{{-}/\mathbb{Z},\log}\simeq (\Omega_{{-}/\Zbb}^m\xrightarrow{1-C^{-1}}\Omega_{{-}/\Zbb}^m/d\Omega_{{-}/\Zbb}^{m-1})$, which is a complex of reciprocity sheaves.
\end{example}

\begin{definition}
	Let us recall the notion of {\it unramified \'etale cohomology} introduced in \cite{OCT89}.
	Let $\Fcal$ be an \'etale sheaf on $\Sm$.
	For a connected $X\in \Sm$, its unramified \'etale cohomology with coefficients in $\Fcal$ is defined by
	$$
		H^n_\ur(X,\Fcal):=\bigcap_{x\in X^{(1)}}\Ker(\del_x\colon H^n_\et(\Spec k(X),\Fcal)\to H^{n+1}_{x,\et}(\Spec \Ocal_{X,x},\Fcal)).
	$$
\end{definition}

Since $H^{n+1}_{x,\et}(\Spec \Ocal_{X,x},\Fcal)\simeq H^{n+1}_{x,\et}(\Spec \Ocal_{X,x}^h,\Fcal)$, an element of $H^n_\et(\Spec k(X),\Fcal)$ lies in $H^n_\ur(X,\Fcal)$ if and only if its image in $H^n_\et(\Spec \Frac \Ocal_{X,x}^h,\Fcal)$ comes from $H^n_\et(\Spec \Ocal_{X,x}^h,\Fcal)$ for any $x\in X^{(1)}$.
In other words, $H^n_\ur(X,\Fcal)$ is isomorphic to the subgroup of $R^n\varepsilon_*\Fcal(k(X))$ consisting of elements whose image in $R^n\varepsilon_*\Fcal(\Frac \Ocal_{X,x}^h)$ come from $R^n\varepsilon_*\Fcal(\Ocal_{X,x}^h)$.
Since $R^n\varepsilon_*\Fcal$ is a Nisnevich sheaf and
$$
	\xymatrix{
		\Spec \Frac\mathcal{O}_{X,x}^h\ar[r]\ar[d]			&\Spec k(X)\ar[d]\\
		\Spec \mathcal{O}_{X,x}^h\ar[r]						&\Spec \mathcal{O}_{X,x}
	}
$$
is a cofiltered limit of elementary Nisnevich squares, we have a Cartesian square
$$
	\xymatrix{
		R^n\varepsilon_*\mathcal{F}(\mathcal{O}_{X,x})\ar[r]\ar[d]			&R^n\varepsilon_*\mathcal{F}(\mathcal{O}_{X,x}^h)\ar[d]\\
		R^n\varepsilon_*\mathcal{F}(k(X))\ar[r]					&R^n\varepsilon_*\mathcal{F}(\Frac\mathcal{O}_{X,x}^h).
	}
$$
Therefore the unramified \'etale cohomology can be rewritten as
$$
	H^n_\ur(X,\Fcal) \simeq \bigcap_{x\in X^{(1)}} \Im (R^n\varepsilon_*\Fcal(\Ocal_{X,x})\to R^n\varepsilon_*\Fcal(k(X))).
$$
In particular, we have
\begin{itemize}
	\item $H^n_\ur(X,\mu_\ell^{\otimes m})\simeq R^n\varepsilon_*(\mu_\ell^{\otimes m})(X)$ for a prime number $\ell$ invertible in $k$, and
	\item $H^n_\ur(X,\Omega^m_{{-}/\mathbb{Z},\log})\simeq R^n\varepsilon_*\Omega^m_{{-}/\mathbb{Z},\log}(X)$ if $\ch(k)=p>0$.
\end{itemize}
These isomorphisms can be extended to a general $X\in \Sm$ by defining $H^n_\ur(X,\Fcal):=\bigoplus_{i=1}^r H^n_\ur(X_i,\Fcal)$ where $X_1,\dots,X_r$ are the connected components of $X$.

\subsection{The category $\CI^{\tau,sp}_\Nis$}\label{def:CIrec}

To prove various theorems on reciprocity sheaves, it is often useful to embed $\RSC_\Nis$ into a certain subcategory of $\bMPST$.

An object $G$ of $\bMPST$ is called {\it semipure} if $G(\mathcal{X})\to G(X^\circ)$ is injective for any modulus pair $\mathcal{X}$.
An object $G$ of $\bMPST$ is said to have {\it $M$-reciprocity} if it is isomorphic to $\tau_!G'$ for some $G'\in \MPST$.
Let $\CIrec$ denote the full subcategory of $\bMNST$ consisting of cube-invariant semipure objects having $M$-reciprocity.

\begin{lemma}[{\cite[Theorem 11.1 and 10.1]{Sai20}}]\label{tilde_in_CIrec}
	For any $F\in \RSC_\Nis$, $\tau_!^\Nis\tilde{F}\in \CIrec$.
\end{lemma}

Note that we have $\underline{\omega}_!\tau_!^\Nis\tilde{F}\simeq F$ for $F\in \RSC_\Nis$ by Lemma \ref{omega_sheafification}.
In other words, we have $\tau_!^\Nis\tilde{F}(X)\simeq F(X)$ for any $X\in \Sm$.
Therefore we can deduce results on $F$ from results on $\tau_!^\Nis\tilde{F}\in \CIrec$.
Here, we quote two important results about $\CIrec$.
A modulus pair $\Xcal$ is called {\it log-smooth} if $X$ is smooth and $D_X$ is a strict normal crossing divisor.
An ambient morphism $\pi\colon \Ycal\to \Xcal $ is called a {\it cube-bundle} if there is a Zariski open covering $X=\bigcup_\lambda U_\lambda$ and an ambient isomorphism $\varphi_\lambda\colon (U_\lambda,D_X|_{U_\lambda})\boxtimes \cube \xrightarrow{\sim} (\pi^{-1}(U_\lambda),D_Y|_{\pi^{-1}(U_\lambda)})$ over $(U_\lambda,D_X|_{U_\lambda})$ for each $\lambda$.

\begin{theorem}[Cube-invariance, {\cite[Theorem 0.6]{Sai20}}]\label{cube_invariance}
	Let $\pi\colon \Ycal\to \Xcal$ be a cube-bundle, where $\Xcal, \Ycal$ are log-smooth modulus pairs.
	Let $G\in \CIrec$.
	Then
	$
		\pi^*\colon H^i_\Nis(\Xcal,G)\to H^i_\Nis(\Ycal,G)
	$
	is an isomorphism for $i\geq 0$.
\end{theorem}

\begin{theorem}[Logarithmic blow-up invariance, {\cite[Theorem 5.1]{Sai}}]\label{log_blowup}
	Let $\Xcal$ be a log-smooth modulus pair and $G\in \CIrec$.
	Suppose that $D_X$ is reduced.
	Let $Z\subset X$ be a smooth closed subscheme which is normal crossing to $D_X$, i.e. for each $x\in D_X$ there is a regular system of parameters $t_1,\dots,t_d$ at $x$ such that $Z=\{t_1=\dots=t_r=0\}$ and $D_X=\{t_1t_2\cdots t_s=0\}$ with $s\leq r$ near $x$.
	Let $\rho\colon \Bl_ZX\to X$ be the projection.
	Then there is an isomorphism
	$$H^i_\Nis(\Xcal,G)\xrightarrow{\sim}H^i_\Nis((\Bl_ZX,(\rho^*D_X)_\red),G)$$
	for $i\geq 0$ which makes the following diagram commutative:
	$$
	\xymatrix{
		&H^i_\Nis((\Bl_ZX,(\rho^*D_X)_\red),G)\ar[d]\\
		H^i_\Nis(\Xcal,G)\ar[r]\ar[ur]^-{\sim}			&H^i_\Nis((\Bl_ZX,\rho^*D_X),G).
	}
	$$
\end{theorem}

\subsection{Modulus fractions}

We introduce the notion of modulus fractions which is a generalization of modulus pairs.
A {\it modulus fraction} (resp. {\it proper modulus fraction}) is an object $(\Ycal\xrightarrow{f}\Xcal)$ of the functor category $\mathop{\mathrm{Fun}}(\Delta^1, \bMSmfin)$ (resp. $\mathop{\mathrm{Fun}}(\Delta^1, \MSmfin)$).
We write $\Xcal/\Ycal$ for $(\Ycal\xrightarrow{f}\Xcal)$  if the ambient morphism $f$ is clear from the context.
A modulus pair $\Xcal$ can be thought of as a modulus fraction $\Xcal/\emptyset$.
The {\it box product} of two modulus fractions are defined by
$$
	\dfrac{\Xcal}{\Ycal}\boxtimes \dfrac{\Xcal'}{\Ycal'} := \dfrac{\Xcal\boxtimes \Xcal'}{(\Xcal\boxtimes \Ycal')\sqcup (\Ycal\boxtimes \Xcal')}.
$$
For a modulus fraction $\Xcal/\Ycal$, we define $\Ztr(\Xcal/\Ycal)\in \bMPST$ by
$$
	\Ztr(\Xcal/\Ycal) := \Coker(\Ztr(\Ycal)\to \Ztr(\Xcal)).
$$
We have $\Ztr((\Xcal/\Ycal)\boxtimes (\Xcal'/\Ycal'))\simeq \Ztr(\Xcal/\Ycal)\otimes^\bMPST \Ztr(\Xcal'/\Ycal')$.
For $G\in \bMPST$ we write $G(\Xcal/\Ycal)$ for $\Hom_\bMPST(\Ztr(\Xcal/\Ycal),G)=\Ker(G(\Xcal)\to G(\Ycal))$, and define $G^{\Xcal/\Ycal}\in \bMPST$ by $G^{\Xcal/\Ycal}:=G((\Xcal/\Ycal)\boxtimes {-})$.
If $G\in \CIrec$ then $G^{\Xcal/\Ycal}\in \CIrec$ by \cite[Lemma 1.5(3)]{BRS}.
If $\Xcal/\Ycal$ is a proper modulus fraction, then we define $\Ztr(\Xcal/\Ycal)\in \MPST$ and $G(\Xcal/\Ycal)$ for $G\in \MPST$ by the same formula.
We simply write $h_0^\cube (\Xcal/\Ycal)$ for $h_0^\cube \Ztr(\Xcal/\Ycal)$ and similarly for $h_0$, $h_0^\Nis$.

\begin{definition}
	We define proper modulus fractions $\Gbb_m^+$, $\Wbb_n^+$ and $\Dbb_n^+$ as follows:
	\begin{align*}
		\Gbb_m^+&:=\dfrac{(\Pbb^1,\{0,\infty\})}{\{1\}},\quad\Wbb_n^+:=\dfrac{(\Pbb^1,(n+1)\{\infty\})}{\{0\}},\quad\Dbb_n^+:=(\Pbb^1,(n+1)\{0,\infty\}).
	\end{align*}
	We also define cube-invariant objects $\Gbb^\#:=h_0^\cube(\Gbb^+)$ for $\Gbb = \Gbb_m,\Wbb_n,\Dbb_n$.
	We write $\Gbb_a^+$ for $\Wbb_1^+$ and $\Gbb_a^\#$ for $\Wbb_1^\#$.
	Note that we have canonical morphisms of modulus fractions
	$$
		\mathbb{W}_n^+\leftarrow \mathbb{D}_n^+\to \mathbb{G}_m^+.
	$$
\end{definition}

\begin{proposition}\label{h0_of_fractions}
	Let $T$ denote the coordinate of $\Abb^1$.
	Let $X$ be an affine smooth scheme.
	\begin{enumerate}
		\item	There is an isomorphism $h_0(\Wbb_n^+)(X)\simeq \Wbb_n(X)$ given by the formula (\ref{rel_pic_fml_1}).
				In particular, we have $\omega_!^\Nis \Wbb_n^\#=h_0^\Nis(\Wbb_n^+)\simeq \Wbb_n$.
		
		\item	There is an isomorphism $h_{0}(\Gbb_m^+)(X)\simeq \Gbb_m(X)$ given by the formula (\ref{rel_pic_fml_2}).
				In particular, we have $\omega_!^\Nis \Gbb_m^\#=h_0^\Nis(\Gbb_m^+)\simeq \Gbb_m$.
		
		\item	There is an isomorphism $h_0(\Dbb_n^+)(X)\simeq \Dbb_n(X)$ which is given by the formula (\ref{rel_pic_fml_3}).
				In particular, we have $\omega_!^\Nis \Dbb_n^\#=h_0^\Nis(\Dbb_n^+)\simeq \Dbb_n$.
	\end{enumerate}
\end{proposition}

\begin{proof}
	(i) For any affine smooth scheme $X$ we have
	$$
		h_0(\Wbb_n^+)(X) = h_0(\Pbb^1,(n+1)\{\infty\})(X)/h_0(\{0\})(X) = H_0^S(\Pbb^1_X/X,(n+1)\{\infty\})/\Zbb[0].
	$$
	By \cite[Theorem 1.1]{RY16}, this is isomorphic to $\Pic(\Pbb^1_X,(n+1)\{\infty\})/\Zbb[0]$ which is isomorphic to $\Wbb_n(X)$ by Proposition \ref{relative_picard} (i).
	(ii) and (iii) can be proved similarly.
\end{proof}

\begin{remark}
	The following diagram is commutative:
	$$
	\xymatrix{
		h_0^\Nis(\mathbb{W}_n^+)	\ar[d]^-\sim	&h_0^\Nis(\mathbb{D}_n^+)\ar[l]\ar[r]\ar[d]^-\sim			&h_0^\Nis(\mathbb{G}_m^+)\ar[d]^-\sim\\
		\mathbb{W}_n			&\mathbb{D}_n\ar@{->>}[l]_-{\pr_2}\ar@{->>}[r]^-{\pr_4}			&\mathbb{G}_m.
	}
	$$
	In particular, the top horizontal arrows are epimorphisms.
\end{remark}

\subsection{Ramification filtration}
In this subsection we recall the ramification filtration on reciprocity sheaves defined and studied in \cite{RS21}.

A {\it geometric discrete valuation field} (or {\it geometric DVF} for short) is a discrete valuation field $(K,v)$ over $k$ such that $\Ocal_v\simeq \Ocal_{X,x}$ holds for some $X\in \Sm$ and $x\in X^{(1)}$.
Similarly, {\it geometric henselian discrete valuation field} (or {\it geometric henselian DVF} for short) is defined to be a discrete valuation field $(L,v)$ over $k$ such that $\Ocal_v$ is a Henselization of some geometric DVF.
For a geometric DVF $(K,v)$ or a geometric henselian DVF $(L,v)$, we define $(\Ocal_v,\mathfrak{m}^{-j}):=(\Spec \Ocal_v, \Spec \Ocal_v/\mathfrak{m}_v^j)\in \bMCor^\pro$.

Let $(L,v)$ be a geometric henselian DVF and $G\in \CI$.
We have a sequence of homomorphisms
$$
\omega_!G(\Ocal_v)=\tau_!G(\Ocal_v)\to \tau_!G(\Ocal_v,\mathfrak{m}_v^{-1})\to \tau_!G(\Ocal_v,\mathfrak{m}_v^{-2})\to \dots\to \tau_!G(L)=\omega_!G(L).
$$
Note that $\omega_!^\Nis G$ is a reciprocity sheaf and hence $\omega_! G(\Ocal_v)\to \omega_! G(L)$ is injective.
We define the {\it conductor associated to $G$} by
$$
	c_L^G\colon \omega_!G(L)\to \Zbb_{\geq 0};\quad a\mapsto \min\{j\mid a\in \Im(\tau_!G(\Ocal_v,\mathfrak{m}_v^{-j})\to \omega_!G(L))\}.
$$
The value $c_L^G(a)$ can be thought of as the ``degree of ramification'' of $a$; $c_L^G(a)=0$ if and only if $a\in \omega_!G(\Ocal_v)$.
The next lemma is immediate from the definition:

\begin{lemma}
	Let $(L,v)$ be a geometric henselian DVF.
	If $\alpha\colon G\to G'$ is a morphism in $\CI$ and $a\in \omega_!G(L)$, then we have
	$c_L^{G'}(\alpha(a))\leq c_L^G(a)$.
\end{lemma}

\begin{lemma}\label{ramification_sum_ineq}
	Let $(L,v)$ be a geometric henselian DVF, $G_1,\dots,G_r\in \CI$ and $a_i\in \omega_!G_i(L)$.
	Write $G=h_0^\cube(\bigotimes_i^\MPST G_i)$.
	Then we have
	$$
		c_L^G(a_1\otimes\dots\otimes a_r)\leq \sum_{i=1}^r c_L^{G_i}(a_i).
	$$
\end{lemma}

\begin{proof}
	Set $j_i:=c_L^G(a_i)$.
	Let $b_i$ be an element of $\tau_!G_i(\Ocal_v, \mathfrak{m}_v^{-j_i})$ whose image in $\omega_!G_i(L)$ is $a_i$.
	Then $b_1\otimes \dots\otimes b_r\in \tau_!(\bigotimes_i^\MPST G_i)(\Ocal_v,\mathfrak{m}_v^{-\sum_{i=1}^r j_i})$, and the claim follows from this.
\end{proof}

If $F$ is a reciprocity sheaf, then we have a canonical lifting $\tilde{F}=h^0_{\cube}\omega^*F \in \CI$ of $F$ and $c_L^{\tilde{F}}$ gives a map $F(L)\to \Zbb_{\geq 0}$.
We simply write $c_L^F$ for this map and call it the {\it motivic conductor} of $F$.
For a proper modulus pair $\mathcal{X}$ and a morphism $\rho\colon \Spec L\to X$ from a geometric henselian DVF, any section $a\in \tilde{F}(\Xcal)$ satisfies the inequality $c_L^F(\rho^*a)\leq v_L(D_X)$.
The next theorem states the converse:

\begin{theorem}[{\cite[Theorem 4.15 (4)]{RS21}}]\label{conductor_to_sections}
	Let $\Xcal$ be a proper modulus pair and $F\in \RSC_\Nis$.
	Then we have
	$$
		\tau_!\tilde{F}(\Xcal) \simeq \biggl\{a\in F(X)\biggm| \begin{array}{l}
		c_L^F(\rho^*a)\leq v_L(D_X)\text{ for any geometric henselian DVF }(L,v)\\
		\text{and any }\rho\colon \Spec L\to X
		\end{array}
		\biggr\}.
	$$
\end{theorem}

$\Abb^1$-invariance of a reciprocity sheaf can be detected by the motivic conductor:

\begin{theorem}[{\cite[Corollary 4.36]{RS21}}]\label{a1_conductor}
	A reciprocity sheaf $F$ is $\Abb^1$-invariant if and only if $c_L^F\leq 1$ holds for any geometric henselian DVF $(L,v)$.
\end{theorem}

\section{Axioms and examples of symbols}

\subsection{General definition}

\begin{definition}\label{axiom}
Let $F\in \RSC_\Nis$ and $G_1,\dots,G_r\in \CI$.
A {\it $(G_1,\dots,G_r)$-symbol for $F$} is a family of multilinear maps
\begin{align}\label{symbol_eq}
	\varphi = \biggl\{\varphi_K\colon \prod_{i=1}^r \omega_!G_i(K)\to F(K)\biggr\}_{K\in \Field_k}
\end{align}
satisfying the following axioms:
\renewcommand{\labelenumi}{(Sym\arabic{enumi})}
\begin{description}
	\item[(Sym1)] $\varphi$ is natural in $K\in \Field_k$.
	\item[(Sym2)] Let $L/K$ be a finite extension in $\Field_k$ and let $f\colon \Spec L\to \Spec K$ denote the induced morphism.
	Then
	$$
		\varphi_K(a_1,\dots,f_*a_i,\dots,a_r)=f_*\varphi_L(f^*a_1,\dots,a_i,\dots,f^*a_r)
	$$
	holds for any $(a_1,\dots,a_r)\in \omega_!G_1(K)\times\dots \times \omega_!G_i(L)\times\dots\times \omega_!G_r(K)$.
	\item[(Sym3)] If $(K,v)$ is a geometric DVF, then the restriction of $\varphi_K$ to $\prod_{i=1}^r\omega_!^\Nis G_i (\Ocal_v)$ factors through $F(\Ocal_v)$.
	The resulting morphism $\varphi_{\Ocal_v}$ makes the following diagram commutative, where $k(v)$ is the residue field of $\mathcal{O}_v$:
	$$
	\xymatrix{
		\prod_{i=1}^r \omega_!G_i^\Nis(\Ocal_v) \ar[r]^-{\varphi_{\Ocal_v}}\ar[d]			&F(\Ocal_v)\ar[d]\\
		 \prod_{i=1}^r \omega_!G_i(k(v)) \ar[r]^-{\varphi_{k(v)}}						&F(k(v)).
	}
	$$
\end{description}
The datum $\varphi$ is called {\it bounded} if moreover the following condition holds:
\begin{description}
	\item[(Sym4)]For any geometric henselian DVF $(L,v)$ and $(a_1,\dots,a_r)\in \prod_{i=1}^r\omega_!G_i(L)$ we have
	$$
		c_L^F(\varphi_L(a_1,\dots,a_r))\leq \sum_{i=1}^r c_L^{G_i}(a_i).
	$$
\end{description}
\end{definition} 

\begin{definition}
	Let $F\in \RSC_\Nis$.
	An {\it (r,s)-symbol for $F$} is a $(\underbrace{\Gbb_a^\#,\dots,\Gbb_a^\#}_{r},\underbrace{\Gbb_m^\#,\dots,\Gbb_m^\#}_s)$-symbol for $F$.
	Since $\omega_!^\Nis{\Gbb_a^\#}\simeq \Gbb_a$ and $\omega_!^\Nis{\Gbb_m^\#}\simeq \Gbb_m$, this is a family of bilinear maps of the form
	$$
		\varphi_K\colon K^r\times (K^\times)^s\to F(K).
	$$
	We use the notation $\varphi_K({-},\dots,{-}|{-},\dots,{-})$ hereinafter.
	The {\it Steinberg conditions} for $\varphi$ are defined as follows:
	\begin{description}
		\item[(ST1)] For any $K\in \Field_k$ and $a\in K^\times\setminus\{1\}$, we have
			$$\varphi_K(\dots|\dots, a,\dots,1-a,\dots)=0.$$
		\item[(ST2)] For any $K\in \Field_k$, $a\in K^\times\setminus\{1\}$ and $c\in K^\times$, we have
			$$\varphi_K(\dots,ca,\dots|\dots,a,\dots)+\varphi_K(\dots,c(1-a),\dots|\dots, 1-a, \dots)=0.$$
		\item[(ST3)] For any $K\in \Field_k$ and $a,b,c,d\in K^\times$, we have
			\begin{align*}&\varphi_K(\dots,bc,\dots,ad,\dots|\dots)+\varphi_K(\dots,ac,\dots,bd,\dots|\dots)\\
					{}={} &\varphi_K(\dots,c,\dots,abd,\dots|\dots)_K + \varphi_K(\dots,abc,\dots,d,\dots|\dots).\end{align*}
	\end{description}	
\end{definition}

Here, the omitted components are assumed to be the same over summands.
The main result of this paper is the following:

\begin{theorem}\label{main}
	Let $\varphi$ be a bounded $(r,s)$-symbol for $F\in \RSC_\Nis$.
	Then $\varphi$ satisfies (ST1).
	If $\ch(k)\neq 2$, then it also satisfies (ST2) and (ST3).
\end{theorem}

\subsection{Examples}

First we give useful criteria for (Sym2) and (Sym4).

\begin{lemma}\label{universal_element}
	Let $F\in \RSC_\Nis$.
	Let $\varphi$ be a datum (\ref{symbol_eq}) for $(\underbrace{\Gbb_a^\#,\dots,\Gbb_a^\#}_{r},\underbrace{\Gbb_m^\#,\dots,\Gbb_m^\#}_s)$ satisfying (Sym1).
	Suppose that there is a universal element $\omega\in F((\Abb^1)^r\times (\Abb^1\setminus\{0\})^s)$, i.e. an element whose pullback by $(a_1,\dots,a_r,b_1,\dots,b_s)\colon \Spec K\to (\Abb^1)^r\times (\Abb^1\setminus\{0\})^s$ equals to $\varphi_K(a_1,\dots,a_r|b_1,\dots,b_s)$.
	Then $\varphi$ satisfies (Sym2).
\end{lemma}

\begin{proof}
	Let $L/K$ be a finite extension and let $f$ denote the induced morphism $\Spec L\to \Spec K$.
	We assume that $r\geq 1$ and we prove
	$$
		\varphi_K(f_*a_1,a_2,\dots,b_s)=f_*\varphi_L(a_1,f^*a_2,\dots,f^*b_s)
	$$
	for $a_1\in L, a_2,\dots,a_r\in K$ and $b_1,\dots,b_s\in K^\times$; other cases are similar.
	We choose subfields $K'\subset K$ and $L'\subset L$ finitely generated over $k$ such that $a_1\in L',\,a_2,\dots,a_r\in K',\,b_1,\dots,b_s\in K'^\times$ and $L\simeq K\otimes_{K'}L'$.
	Replacing $K$ by $K'$ and $L$ by $L'$, we may assume from the beginning that $K$ and $L$ are finitely generated over $k$.
	
	Now choose connected smooth schemes $X$ and $Y$ such that $k(X)\simeq K$ and $k(Y)\simeq L$.
	Shrinking $X$ and $Y$ if necessary, we may assume that $a_1\in \Ocal_Y(Y)$, $a_2,\dots,a_r\in \Ocal_X(X)$, $b_1,\dots,b_r\in \Ocal_X(X)^\times$ and that there is a finite surjective morphism $f\colon Y\to X$ inducing $K\subset L$.
	Consider the following commutative diagram in $\Cor$:
	$$
		\xymatrix@C=90pt{
			X\ar[r]^-{({}^tf,\id,\dots,\id)}\ar[d]	^-{{}^tf}			&Y\times X\times \dots\times X\ar[r]	^-{a_1\times \dots\times a_r\times b_1\times \dots \times b_s}	&(\Abb^1)^r\times (\Abb^1\setminus\{0\})^s\\
			Y\ar[r]^-{(\id\dots,\id)}					&Y\times Y\times \dots \times Y\ar[u]^-{\id \times f\times \dots \times f}.
		}
	$$
	Comparing two ways for pulling back $\omega$ to $X$, we get
	$$
		(f_*a_1,a_2,\dots,b_s)^*\omega = f_*(a_1,f^*a_2,\dots,f^*b_s)^*\omega \in F(X).
	$$
	Restricting this equality to the function field, we get (Sym2).
\end{proof}

\begin{lemma}
	Let $G_1,\dots,G_r\in \CI$ and let $\varphi$ be a $(G_1,\dots,G_r)$-symbol for $F\in \RSC_\Nis$.
	If $F$ is $\Abb^1$-invariant, then $\varphi$ is bounded.
\end{lemma}

\begin{proof}
	Let $(L,v)$ be a geometric henselian DVF.
	Since $F$ is $\mathbb{A}^1$-invariant, we have $c_L^F\leq 1$ by Theorem \ref{a1_conductor}.
	Therefore (Sym4) is equivalent to saying that $(a_1,\dots,a_r)\in \prod_{i=1}^r \omega_!G_i(\Ocal_v)$ implies $\varphi_K(a_1,\dots,a_r)\in F(\Ocal_v)$.
	We can write $\mathcal{O}_v$ as a filtered colimit of geometric DVFs, so the claim follows from (Sym3).
\end{proof}

We give some examples of $(0,s)$-symbols.

\begin{example}
	Consider the Milnor $K$-theory sheaf $K^M_s\in \RSC_\Nis$.
	We define the {\it canonical symbol} by 
	$$
		\varphi_K\colon (K^\times)^s \to K^M_s(K);\quad(a_1,\dots,a_s)\mapsto \{a_1,\dots,a_s\}.
	$$
	It clearly satisfies (Sym1) and (Sym3), and (Sym2) follows from the projection formula for the Milnor $K$-theory.
	Moreover, it is bounded since $K^M_s$ is $\mathbb{A}^1$-invariant.
	This symbol satisfies (ST1) by definition.
\end{example}

\begin{example}\label{differential}
	Consider the sheaf of logarithmic absolute differential forms $\Omega^s_{{-}/\Zbb,\log} \in \RSC_\Nis$.
	We define the {\it differential symbol} by
	$$
		\varphi_K\colon (K^\times)^s \to \Omega^s_{K/\Zbb,\log};\quad(a_1,\dots,a_s)\mapsto \dlog(a_1)\wedge \dots\wedge \dlog(a_s).
	$$
	It clearly satisfies (Sym1) and (Sym3), and (Sym2) follows from the existence of a universal element $\omega = \dlog(X_1)\wedge \dots\wedge \dlog(X_s) \in \Omega^s_{k[X_1^\pm,\dots,X_s^\pm]/\Zbb,\log}$.
	Moreover, it is bounded since $\Omega^s_{{-}/\Zbb,\log}$ is $\Abb^1$-invariant.
	This symbol satisfies (ST1):
	$$
		\dlog(a)\wedge \dlog(1-a)=\dfrac{-da\wedge da}{a(1-a)}=0.
	$$
\end{example}

\begin{example}
	Let $\ell$ be a prime number invertible in $k$.
	Consider the unramified \'etale cohomology sheaf $H^s_\ur({-},\mu_\ell^{\otimes s}) \in \RSC_\Nis$.
	The {\it norm residue symbol} is defined by
	$$
		\varphi_K\colon (K^\times)^s \to H^s_\ur(\Spec K, \mu_\ell^{\otimes s}) = H^s_\et(\Spec K, \mu_\ell^{\otimes s});\quad (a_1,\dots,a_s)\mapsto \del a_1\cup \dots \cup \del a_s
	$$
	where $\del \colon H^0_\et (\Spec K, \Gbb_m) \to H^1_\et (\Spec K, \mu_\ell)$ is the boundary homomorphism arising from the Kummer sequence.
	It clearly satisfies (Sym1) and (Sym3), and (Sym2) follows from the existence of a universal element $\del X_1\cup \dots \cup \del X_s \in H^s_\ur ((\Abb^1\setminus\{0\})^s, \mu_\ell^{\otimes s})$.
	Moreover, it is bounded since $H^s_\ur({-}, \mu_\ell^{\otimes s})$ is $\Abb^1$-invariant by Theorem \ref{a1_invariant_sheafification}.
	This symbol satisfies (ST1); see \cite[Theorem 3.1]{Tat76}.
\end{example}

Next we give examples of $(1,s)$-symbols.

\begin{example}\label{extended_differential}
	Consider the sheaf of absolute differential forms $\Omega^s_{{-}/\Zbb} \in \RSC_\Nis$.
	We define the {\it extended differential symbol} by
	$$
		\varphi\colon K\times (K^\times)^s \to \Omega^s_{K/\Zbb};\quad(a,b_1,\dots,b_s)\mapsto a\dlog(b_1)\wedge\dots\wedge \dlog(b_s).
	$$
	It clearly satisfies (Sym1) and (Sym3), and (Sym2) follows from the existence of a universal element $X\dlog(Y_1)\wedge\dots\wedge\dlog(Y_s)\in \Omega^s_{k[X,Y_1^\pm,\dots,Y_s^\pm]/\Zbb}$.
	We will prove later that it is bounded.
	This symbol satisfies (ST1) and (ST2):
	$$
		ca\dlog(a)+c(1-a)\dlog(1-a)=cda+cd(1-a)=0.
	$$
\end{example}

\begin{example}
	Suppose that $\ch(k)=p>0$.
	Consider the unramified \'etale cohomology sheaf $H^1_{\ur}({-},\Omega^s_{{-}/\mathbb{Z},\log}) \in \RSC_\Nis$.
	We define the {\it $p$-norm residue symbol} by
	\begin{align*}
		\varphi_K\colon K\times (K^\times )^s&\to H^1_\ur(\Spec K, \Omega^s_{{-}/\mathbb{Z},\log})=H^1_\et(\Spec K, \Omega^s_{{-}/\mathbb{Z},\log})\\
		\quad (a,b_1,\dots,b_s) &\mapsto \chi_a\cup \dlog (b_1)\wedge \dots\wedge \dlog (b_s)
	\end{align*}
	where $\chi_a \in H^1_\et (\Spec K, \Zbb/p) \simeq \Hom_{\rm cont}(\Gal(K^{\rm sep}/K),\Zbb/p)$ is the Artin-Schreier character of the extension $K(\alpha)/K$ where $\alpha^p-\alpha=a$.
\end{example}

\begin{lemma}
	Suppose that $\ch(k)=p>0$ and let $K$ be a field over $k$.
	Then the composition
	$$
		K\times (K^\times)^s\xrightarrow{\psi} \Omega^s_{K/\mathbb{Z}}\twoheadrightarrow \Omega^s_{K/\mathbb{Z}}/d\Omega^{s-1}_{K/\mathbb{Z}}\xrightarrow{\del}H^1_\et(\Spec K, \Omega^s_{{-}/\mathbb{Z},\log})
	$$
	coincides with the $p$-norm residue symbol, where $\psi$ is the extended differential symbol and $\del$ is the boundary homomorphism arising from the sequence (\ref{Cartier_sequence}).
	In particular, the $p$-norm residue symbol is bounded and satisfies (ST1) and (ST2).
\end{lemma}

\begin{proof}
	Let $b_1,\dots,b_s\in K^\times$.
	Then the multiplication by $\dlog(b_1)\wedge\dots\wedge \dlog(b_s)$ gives a morphism
	$\mathbb{G}_a\to \Omega^s_{K/\mathbb{Z}}$
	of \'etale sheaves on $\Spec K$.
	This induces a morphism
	$$
		\xymatrix{
			0\ar[r]			&\mathbb{Z}/p\ar[r]\ar[d]		&\mathbb{G}_a\ar[r]^-{a\mapsto a^p-a}\ar[d]		&\mathbb{G}_a\ar[r]\ar[d]			&0\\
			0\ar[r]			&\Omega^s_{{-}/\mathbb{Z},\log}\ar[r]		&\Omega^s_{{-}/\mathbb{Z}}\ar[r]^-{C^{-1}}&\Omega^s_{{-}/\mathbb{Z}}/d\Omega^{s-1}_{{-}/\mathbb{Z}}\ar[r]		&0
		}
	$$
	of exact sequence of \'etale sheaves on $\Spec K$, and hence yields a commutative diagram
	$$
		\xymatrix{
			K\ar[r]^-{a\mapsto \chi_a}\ar[d]_-{\wedge \dlog(b_1)\wedge\dots\wedge\dlog(b_s)}		&H^1_\et(\Spec K, \mathbb{Z}/p)\ar[d]^-{\wedge \dlog(b_1)\wedge\dots\wedge\dlog(b_s)}\\
			\Omega^s_{K/\mathbb{Z}}/d\Omega^{s-1}_{K/\mathbb{Z}}\ar[r]^-{\partial}	&H^1_\et(\Spec K, \Omega^s_{{-}/\mathbb{Z},\log}).
		}
	$$
	This proves the claim.
\end{proof}

The extended differential symbol has the following universal property:

\begin{lemma}\label{omega_universal}
	Let $\varphi$ be a $(1,s)$-symbol for $F\in \RSC_\Nis$.
	Then it satisfies (ST1) and (ST2) if and only if each $\varphi_K$ factors through the extended differential symbol $K\times (K^\times)^s\to \Omega^s_{K/\mathbb{Z}}$.
\end{lemma}

This follows from the following
\begin{lemma}
	The kernel of the map $K\otimes (K^\times)^{\otimes s}\to \Omega^s_{K/\Zbb};\;a\otimes b_1\otimes \dots\otimes b_s\mapsto a\dlog(b_1)\wedge\dots\wedge \dlog(b_s)$ is generated by the elements of the form
	\begin{align}
		\label{st1}&a\otimes\cdots\otimes b\otimes \dots\otimes (1-b)\otimes \cdots,\\
		\label{st2}&(a\otimes a\otimes \cdots) +((1-a)\otimes (1-a)\otimes \cdots).
	\end{align}

\end{lemma}

\begin{proof}
	Let $H$ denote the kernel of the given map.
	We have already seen in Example \ref{differential} and Example \ref{extended_differential} that (\ref{st1}) and (\ref{st2}) belong to $H$.
	On the other hand, according to \cite[Lemma 4.1]{BE03}, $\Omega^s_{K/\mathbb{Z}}$ is the quotient of $K\otimes \bigwedge^sK^\times$ by elements of the form (\ref{st2}).
	Therefore it suffices to check that the kernel of $K\otimes (K^\times)^{\otimes s}\to K\otimes \bigwedge^s K^\times$ is contained in the subgroup generated by elements of the form (\ref{st1}).
	This is equivalent to saying that $a\otimes \{b,b\}=0$ holds in $K\otimes K_2^M(K)$.
	Note that we always have $\{b,b\}=\{b,-1\}$ in $K_2^M(K)$.
	If $\ch(k)=2$, then we have $\{b,-1\}=\{b,1\}=0$ in $K_2^M(K)$.
	If $\ch(k)\neq 2$, then we have
	$$
		a\otimes \{b,-1\}=\dfrac{a}{2}\otimes \{b,1\}=0
	$$
	in $K\otimes K_2^M(K)$.
	This proves the claim.
\end{proof}

The following is an example of a $(2,0)$-symbol.

\begin{example}\label{infinitesimal}
	Consider the first infinitesial neighborhood of the diagonal $\Pcal^1 \in \RSC_\Nis$.
	We define the {\it infinitesimal symbol} by
	$$
		\varphi_K\colon K\times K\to \Pcal^1(K);\quad(a,b)\mapsto [a\otimes b].
	$$
	It clearly satisfies (Sym1) and (Sym3), and (Sym2) follows from the existence of a universal element $[X\otimes Y]\in \Pcal^1(\mathbb{A}^2)$.
	We will prove later that it is bounded.
	It satisfies (ST3) by definition.
\end{example}

\section{Sheaf-theoretic interpretation}

In this section we provide a sheaf-theoretic interpretation for $(G_1,\dots,G_r)$-symbols.
First we note an obvious result about exterior products:

\begin{lemma}\label{tensor_lemma1}
	Let $G_1,\dots,G_r,F \in \NST$.
	Then a morphism $\alpha\colon \bigotimes_i^\NST G_i\to F$ induces, by the exterior product, a family of multilinear maps
	\begin{align}\label{tensor_lemma_eq1}
		\biggl\{\alpha_{X_1,\dots,X_r}\colon \prod_{i=1}^r G_i(X_i) \to F\biggl(\prod_{i=1}^r X_i\biggr)\biggr\}_{X_1,\dots,X_r\in \Sm}
	\end{align}
	with the following property:
	\begin{description}
		\item[(Ext)]	If $g\in \Cor(X_i, Y_i)$ is a finite correspondence, then we have
		$$\alpha_{X_1,\dots, X_r}(a_1,\dots,g^*a_i,\dots,a_r)=(\id\times \dots\times g\times \dots\times \id)^*\alpha_{X_1,\dots,Y_i,\dots,X_r}(a_1,\dots,a_r)$$
		for any $(a_1,\dots,a_r)\in G_1(X_1)\times\dots\times G_i(Y_i)\times \dots\times G_r(X_r)$.
	\end{description}
	Conversely, any such datum comes from a unique morphism $\alpha\colon \bigotimes_i^\NST G_i\to F$.
\end{lemma}

Our starting point is the following variation of this result; roughly speaking, it says that a morphism $\alpha\colon \bigotimes_i^\NST F_i\to G$ corresponds to a natural multilinear operation $\prod_{i=1}^r F_i(X) \to G(X)$ satisfying the ``projection formula''.

\begin{lemma}\label{tensor_operation}
	Let $G_1,\dots,G_r,F \in \NST$ and suppose that $F$ has global injectivity.
	Then a morphism $\alpha\colon \bigotimes_i^\NST G_i\to F$ induces, by the tensor product, a family of multilinear maps
	\begin{align}\label{tensor_lemma_eq2}
		\biggl\{\alpha_X\colon \prod_{i=1}^r G_i(X) \to F(X)\biggr\}_{X \in \Sm}
	\end{align}
	with the following properties:
	\begin{description}
		\item[(Ten1)] It is natural in $X\in \Sm$.
		\item[(Ten2)] If $f\colon Y\to X$ is a finite surjective morphism in $\Sm$, then we have
		$$
			\alpha_X(a_1,\dots,f_*a_i,\dots,a_r)=f_*\alpha_Y(f^*a_1,\dots,a_i,\dots, f^*a_r)
		$$
		for any $(a_1,\dots,a_r)\in G_1(X)\times\dots\times G_i(Y)\times \dots\times G_r(X)$.
	\end{description}
	Conversely, any such datum comes from a unique morphism $\alpha\colon \bigotimes_i^\NST G_i\to F$.
\end{lemma}

\begin{proof}
	We prove that a datum (\ref{tensor_lemma_eq1}) satisfying (Ext) is equivalent to a datum (\ref{tensor_lemma_eq2}) satisfying (Ten1) and (Ten2).
	Consider a datum (\ref{tensor_lemma_eq1}) satisfying (Ext).
	First we prove that
	$\alpha_X(a_1,\dots,a_r) := \Delta_X^*\alpha_{X,\dots,X}(a_1,\dots,a_r)$
	gives a datum (\ref{tensor_lemma_eq2}) satisfying (Ten1) and (Ten2).
	(Ten1) can be easily deduced from (Ext).
	To prove (Ten2), consider the following commutative diagram in $\Cor$:
	$$
		\xymatrix@C=60pt{
			X\ar[d]^-{{}^tf}\ar[r]^-{\Delta_X}		&X\times\dots \times X\ar[d]^-{\id\times \dots\times {}^tf\times \dots\times \id}\\
			Y\ar[r]^-{(f,\dots,\id,\dots,f)}					&X\times \dots\times Y\times \dots \times X.
		}
	$$
	Using (Ext) we can calculate as follows:
	\begin{align*}
			&\alpha_X(a_1,\dots,f_*a_i,\dots,a_r)\\
		=	&\Delta_X^*(\id\times\dots\times f\times\dots\times\id)_*\alpha_{X,\dots,Y,\dots,X}(a_1,\dots,a_i,\dots,a_r)\\
		=	&f_*(f,\dots,\id,\dots,f)^*\alpha_{X,\dots,Y,\dots,X}(a_1,\dots,a_i,\dots,a_r)\\
		=	&f_*\alpha_Y(f^*a_1,\dots,a_i,\dots,f^*a_r).
	\end{align*}
	
	Conversely, consider a datum (\ref{tensor_lemma_eq2}) satisfying (Ten1) and (Ten2).
	We will prove that
	$\alpha_{X_1,\dots,X_r}(a_1,\dots,a_r) := a_{X_1\times\dots\times X_r}(\pr_1^*a_1,\dots,\pr_r^*a_r)$
	gives a datum (\ref{tensor_lemma_eq1}) satisfying (Ext).
	We may assume $g=[V]$ where $V$ is an elementary correspondence from $X_i$ to $Y_i$.
	Moreover, we may assume that $X_i$ and $Y_i$ are connected.
	Since $k$ is perfect, there is a dense open subset $U_i\subset X_i$ such that $p^{-1}(U_i)$ is smooth, where $p\colon V\to X_i$ is the projection.
	Let $j\colon U_i\to X_i$ be the inclusion and $q\colon p^{-1}(U_i)\to Y_i$ the projection.
	Then we have the following commutative diagram in $\Cor$:
	$$
		\xymatrix{
			U_i\ar[r]^-j	\ar[d]^-{{}^tp}			&X_i\ar[d]^-{g}\\
			p^{-1}(U_i)\ar[r]^-q				&Y_i.
		}
	$$
	Using (Ten1) and (Ten2) we can calculate as follows:
	\begin{align*}
			&(\id\times\dots\times j\times \dots\times \id)^*\alpha_{X_1,\dots,X_r}(a_1,\dots,g^*a_i,\dots,a_r)\\
		=	&\alpha_{X_1\times\dots\times U_i\times \dots \times X_r}(\pr_1^*a_1,\dots,\pr_i^*j^*g^*a_i,\dots,\pr_r^*a_r)\\
		=	&\alpha_{X_1\times\dots\times U_i\times \dots \times X_r}(\pr_1^*a_1,\dots,\pr_i^*p_*q^*a_i,\dots,\pr_r^*a_r)\\
		=	&(\id\times\dots\times p\times \dots\times \id)_*(\id\times\dots\times q\times \dots \times \id)^*\alpha_{X_1,\dots,Y_i,\dots X_r}(a_1,\dots,a_r)\\
		=	&(\id\times\dots\times j\times \dots\times \id)^*(\id\times\dots\times g\times \dots\times \id)^*\alpha_{X_1,\dots,Y_i,\dots X_r}(a_1,\dots,a_r).
	\end{align*}
	Since $F$ has global injectivity, $(\id\times\dots\times j\times \dots\times \id)^*$ is injective and hence (Ext) holds.
\end{proof}

\begin{lemma}\label{geometric_symbol}
	Let $G_1,\dots,G_r \in \CI$ and let $\varphi$ be a $(G_1,\dots,G_r)$-symbol for $F\in \RSC_\Nis$.
	Then $\varphi$ extends uniquely to a morphism $\bigotimes^\NST\omega_!^\Nis G_i\to F$ in $\NST$.
\end{lemma}

\begin{proof}
	By Lemma \ref{tensor_operation}, it suffices to prove that there is a family of multilinear maps
	$$
		\biggl\{\varphi_X\colon \prod_{i=1}^r \omega_!^\Nis G_i(X)\to F(X)\biggr\}_{X\in \Sm}
	$$
	satisfying (Ten1) and (Ten2).
	Multilinearity of $\varphi_X$ and (Ten2) follows automatically from multilinearity of $\varphi$ and (Sym2) since $F$ has global injectivity.
	Hence we have only to construct a datum with (Ten1), i.e. a morphism $\prod_{i=1}^r \omega_!^\Nis G_i \to F$ of sheaves of sets.
	Since each $\omega_!^\Nis G_i$ is in $\RSC_\Nis$, $\prod_{i=1}^r \omega_!^\Nis G_i$ is an unramified sheaf of sets.
	Therefore, by Theorem \ref{unramified_lemma_Morel}, it suffices to show that for any $X\in \Sm$ and any $x\in X^{(1)}$, the image of $\prod_{i=1}^r\omega_!^\Nis G_i (\Ocal_{X,x})$ under $\varphi_{\Frac \Ocal_{X,x}}$ is contained in $F(\Ocal_{X,x})$ and the diagram
	$$
	\xymatrix{
		\prod_{i=1}^r\omega_!^\Nis G_i(\Ocal_{X,x})\ar[r]^-{\varphi_{\Ocal_{X,x}}}\ar[d]		&F(\Ocal_{X,x})\ar[d]\\
		\prod_{i=1}^r\omega_! G_i (k(x))\ar[r]^-{\varphi_{k(x)}}					&F(k(x))
	}
	$$
	is commutative.
	This is guaranteed by (Sym3).
\end{proof}

Conversely, it is easy to see that any morphism $\bigotimes^\NST\omega_!^\Nis G_i\to F$ in $\NST$ gives rise to a $(G_1,\dots,G_r)$-symbol for $F$.
We will use this one-to-one correspondence implicitly hereinafter.
Next we will characterize bounded symbols in terms of the corresponding morphisms in $\NST$.

\begin{definition}
	Let $G_1,\dots,G_r\in \CI$ and $K\in \Field_k$.
	A {\it Weil datum} for $G_1,\dots,G_r$ over $K$ consists of
	\begin{enumerate}
		\item a connected regular projective curve $C$ over $K$,
		\item effective Cartier divisors $D_1,\dots,D_r$ on $C$,
		\item $a_i\in \tau_!G_i(C,D_i)$ for $i=1,\dots,r$, and
		\item a finite $K$-morphism $f\colon C\to \mathbb{P}^1_K$ with $f\equiv 1\mod D:=D_1+\dots+D_r$.
	\end{enumerate}
	We say that a symbol $\varphi \colon \bigotimes^\NST \omega_!^\Nis G_i\to F$ satisfies the {\it Weil reciprocity} if for any $K\in \Field_k$ and any Weil datum for $G_1,\dots,G_r$ over $K$, we have
	$$
		\sum_{x\in (C\setminus |D|)^{(1)}}v_x(f)\cdot \Tr_{k(x)/K}\varphi_{k(x)}(a_1(x), \dots, a_r(x))=0.
	$$
\end{definition}

Now observe that $h_0^\Nis(\bigotimes_{i}^\MPST G_i)$ is a quotient of $\bigotimes^\NST\omega_!^\Nis G_i$; indeed, we have $h_0^\Nis(\bigotimes_{i}^\MPST G_i)=\omega_!^\Nis h_0^\cube(\bigotimes_{i}^\MPST G_i)$ and $\bigotimes^\NST\omega_!^\Nis G_i\simeq \omega_!^\Nis(\bigotimes_{i}^\MPST G_i)$.

\begin{theorem}\label{bounded_Weil}
	Let $F\in \RSC_\Nis$, $G_1,\dots,G_r \in \CI$ and $\varphi\colon \bigotimes^\NST\omega_!^\Nis G_i\to F$.
	Then the following conditions are equivalent:
	\begin{enumerate}
	\item $\varphi$ factors through $h_0^\Nis(\bigotimes_{i}^\MPST G_i)$.
	\item $\varphi$ is bounded.
	\item $\varphi$ satisfies the Weil reciprocity.
	\end{enumerate}
\end{theorem}

\begin{proof}
	(i)$\implies$(ii) The morphism $h_0^\Nis(\bigotimes_{i}^\MPST G_i)\to F$ corresponds to a morphism $\bigotimes_{i}^\MPST G_i\to \tilde{F}$ in $\MPST$ by adjunction.
	The claim follows from this and Lemma \ref{ramification_sum_ineq}.
	
	(ii)$\implies$(iii) Let $K\in \Field_k$ and consider a Weil datum over $K$.
	Then we have $\varphi(a_1\otimes \dots\otimes a_r)\in \tau_!\tilde{F}(C,D)$ by boundedness and Theorem \ref{conductor_to_sections}.
	The condition for $f\colon C\to \mathbb{P}^1$ means that ${}^tf$ gives an element of $\bMCor^\pro((\Pbb^1_K,\{1\}),(C,D))$ and we have
	$$
		\sum_{x\in (C\setminus |D|)^{(1)}}v_x(f)\cdot \Tr_{k(x)/K}\varphi_{k(x)}(a_1(x),\dots,a_r(x))= (i_0^*-i_\infty^*)({}^tf)^*\varphi(a_1\otimes \dots \otimes a_r)
	$$
	(see the proof of \cite[Theorem 4.7]{RSY}).
	The last term is $0$ by the cube-invariance of $\tilde{F}$.
	
	(iii)$\implies$(i) By Lemma \ref{injectivity_factorization_lemma}, it suffices to show that for each $K\in \Field_k$, $\varphi_K$ factors through $h_0^\Nis(\bigotimes_{i}^\MPST G_i)(K)$.
	Let $H_K$ be the kernel of the canonical surjection $(\bigotimes_{i}^\NST \omega_!^\Nis{G_i})(K) \to h_0^\Nis(\bigotimes_{i}^\MPST G_i)(K)$.
	By \cite[Theorem 4.9]{RSY}, $H_K$ is generated by elements of the form
	$$
		\sum_{x\in (C\setminus |D|)^{(1)}} v_x(f)\cdot \Tr_{k(x)/K}(a_1(x)\otimes \dots\otimes a_r(x))
	$$
	associated to a Weil datum over $K$.
	Hence we get $\varphi_K(H_K)=0$ by the Weil reciprocity.
\end{proof}

\begin{remark}
	The equivalence of (i) and (iii) in Theorem \ref{bounded_Weil} holds for a general $F\in \NST$ having global injectivity.
	Indeed, the proof of (iii)$\implies$(i) applies without any change, and the converse is also clear from the proof.
\end{remark}

\begin{example}\label{extended_differential_bounded}
	We prove that the differential symbol $\varphi_K\colon K\times (K^\times)^s\to \Omega^s_{{-}/\Zbb}$ is bounded (see Example \ref{extended_differential} for the definition).
	By Theorem \ref{bounded_Weil}, it suffices to prove that $\varphi$ satisfies the Weil reciprocity, i.e.
	$$
	\sum_{x\in (C\setminus |D|)^{(1)}} v_x(f)\cdot \Tr_{k(x)/K}(a(x)\dlog(b_1(x))\wedge\dots\wedge\dlog(b_s(x)))=0
	$$
	holds for any Weil datum $(C,(D,E_1,\dots,E_s),(a,b_1,\dots,b_s),f)$ over $K$.
	
	Fix a point $x\in |D|$ and let $R=\mathcal{O}_{C,x}^h$, $S=\Spec R$ and $U=\Spec \Frac R$.
	We prove that $a\dlog(b_1)\wedge\dots\wedge \dlog(b_s)\wedge\dlog(f)$ is regular on $S$.
	Let $z$ denote the closed point of $S$ and $d,e_1,\dots,e_s$ be the multiplicities of $D,E_1,\dots,E_s$ at $z$.
	Regard $a$ (resp. $b_i$) as an element of $\tau_!\mathbb{G}_a^\#(S,d\{z\})$ (resp. $\tau_!\mathbb{G}_m^\#(S,e_i\{z\})$) via Lemma \ref{h0_of_fractions}.
	Replacing $S$ by its finite extension, we may assume that $a$ (resp. $b_i$) is represented by an ambient morphism $(S,d\{z\})\to (\mathbb{P}_K,2\{\infty\})$ (resp. $(S,e_i\{z\})\to (\mathbb{P}_K,\{0,\infty\})$).
	
	We may assume that $a\neq 0$.
	Choose a uniformizer $\pi$ of $R$ and write $a=u_a\pi^{-m_a}$, $b_i=u_{b_i}\pi^{m_{b_i}}$ and $f=1+u_f\pi^{m_f}$ with $u_a,u_{b_i},u_f\in R^\times, m_a,m_{b_i}\in \mathbb{Z}, m_f\in \mathbb{Z}_{\geq 1}$.
	By the above interpretation of $a,b_1,\dots,b_i$ as ambient morphisms, we get
	$$
		d\geq 2\max\{m_a,0\},\quad e_i\geq |m_{b_i}|.
	$$
	Let $P=\Spec K[T]_{(T-1)}^h$ and $p$ be its closed point.
	Our assumption on $f$ implies that ${}^tf$ gives an admissible finite correspondence $(P,\{p\})\to (S,(d+e_1+\dots+e_s)\{z\})$.
	Hence we have
	$$
		m_f\geq d+e_1+\dots+e_s.
	$$
	These inequalities imply $m_f\geq 2\max\{m_a,0\}+\sum_{i=1}^s|m_{b_i}|\geq m_a+1$ except for $m_a=0$.
	However, this inequality is true also for $m_a=0$ since $m_f$ is always positive.
	Now we have
	\begin{align*}
		&a\dlog(b_1)\wedge\dots\wedge \dlog(b_s)\wedge \dlog (f)\\
		=&\dfrac{u_a\pi^{m_f-m_a-1}}{(1+u_f\pi^{m_f})}\cdot\biggl(\dfrac{du_{b_1}}{u_{b_1}}+\dfrac{m_{b_1}d\pi}{\pi}\biggr)\wedge\dots\wedge\biggl(\dfrac{du_{b_s}}{u_{b_s}}+\dfrac{m_{b_s}d\pi}{\pi}\biggr)\wedge(\pi du_f+m_fu_fd\pi)
	\end{align*}
	and this is regular on $S$ since $m_f-m_a-1\geq 0$.
	This proves the claim.
	Using this result, we get
	\begin{align*}
	&\sum_{x\in (C\setminus |D|)^{(1)}} v_x(f)\cdot \Tr_{k(x)/K}(a(x)\dlog(b_1(x))\wedge\dots\wedge\dlog(b_s(x)))\\
	=&\sum_{x\in (C\setminus |D|)^{(1)}} \Res_{C/K,x}(a\dlog(b_1)\wedge\dots\wedge\dlog(b_s)\wedge\dlog(f))\\
	=&\sum_{x\in C^{(1)}} \Res_{C/K,x}(a \dlog(b_1)\wedge\dots\wedge\dlog(b_s)\wedge\dlog(f))=0
	\end{align*}
	by the residue theorem (see \cite[\S17]{Kun86} for the basic properties of the residue map).
\end{example}

\begin{example}\label{infinitesimal_bounded}
	We prove that the infinitesimal symbol $\varphi_K\colon K\times K\to \Pcal^1$ is bounded (see Example \ref{infinitesimal} for the definition).
	Here we identify $\Pcal^1$ with $\Gbb_a\oplus \Omega_{{-}/\Zbb}^1$ by $[a\otimes b]\mapsto (ab,adb)$.
	By Theorem \ref{bounded_Weil}, it suffices to prove that $\varphi$ satisfies the Weil reciprocity, i.e.
	$$
	\sum_{x\in (C\setminus |D|)^{(1)}} v_x(f)\cdot \Tr_{k(x)/K}(a(x)b(x),a(x)db(x))=0
	$$
	holds for any Weil datum $(C,(D,E),(a,b),f)$ over $K$.
	
	Fix a point $x\in |D|$ and let $R=\mathcal{O}_{C,x}^h$, $S=\Spec R$ and $U=\Spec \Frac R$.
	We prove that $(ab\dlog(f),adb\wedge\dlog(f))$ is regular on $S$.
	Let $z$ denote the closed point or $S$ and $d,e$ be the multiplicities of $D,E$ at $z$.
	Regard $a$ (resp. $b$) as an element of $\tau_!\mathbb{G}_a^\#(S,d\{z\})$ (resp. $\tau_!\mathbb{G}_a^\#(S,e\{z\})$) via Lemma \ref{h0_of_fractions}.
	Replacing $S$ by its finite extension, we may assume that $a$ (resp. $b$) is represented by an ambient morphism $(S,d\{z\})\to (\mathbb{P}_K,2\{\infty\})$ (resp. $(S,e\{z\})\to (\mathbb{P}_K,2\{\infty\})$).

	We may assume that $ab\neq 0$.
	Choose a uniformizer $\pi$ of $R$ and write $a=u_a\pi^{-m_a}$, $b=u_b\pi^{-m_b}$ and $f=1+u_f\pi^{m_f}$ with $u_a,u_b,u_f\in R^\times, m_a,m_b\in \mathbb{Z}, m_f\in \mathbb{Z}_{\geq 1}$.
	By the above interpretation of $a,b$ as ambient morphisms, we get
	$$
		d\geq 2\max\{m_a,0\},\quad e\geq 2\max\{m_b,0\}.
	$$
	Let $P=\Spec K[T]_{(T-1)}^h$ and $p$ be its closed point.
	Our assumption on $f$ implies that ${}^tf$ gives an admissible finite correspondence $(P,\{p\})\to (S,(d+e)\{z\})$.
	Hence we have
	$$
		m_f\geq d+e.
	$$
	These inequalities imply $m_f\geq 2\max\{m_a,0\}+2\max\{m_b,0\}\geq m_a+m_b+1$ except for $m_a=m_b=0$.
	However, this inequality is true also for $m_a=m_b=0$ since $m_f$ is always positive.
	Now we have
	\begin{align*}
			&(ab\dlog(f),adb\wedge\dlog(f))\\
			=&\dfrac{u_a u_b\pi^{m_f-m_a-m_b-1}}{1+u_f\pi^{m_f}}(\pi du_f+m_fu_f d\pi, d\pi\wedge du_f)
	\end{align*}
	and this is regular on $S$.
	This proves the claim.
	Using this result, we get
	\begin{align*}
	&\sum_{x\in (C\setminus |D|)^{(1)}} v_x(f)\cdot \Tr_{k(x)/K}(a(x)b(x),a(x)db(x))\\
	=&\sum_{x\in (C\setminus |D|)^{(1)}} (\Res_{C/K,x}ab\dlog(f),\Res_{C/K,x}adb\wedge\dlog(f)))\\
	=&\sum_{x\in C^{(1)}} (\Res_{C/K,x}ab\dlog(f),\Res_{C/K,x}adb\wedge\dlog(f))=0
	\end{align*}
	by the residue theorem.
\end{example}

We regard the canonical/extended differential/infinitesimal symbols as morphisms
\begin{align*}
	&h_0^\Nis((\Gbb_m^+)^{\boxtimes s}) \to K^M_s,\\
	&h_0^\Nis(\Gbb_a^+\boxtimes (\Gbb_m^+)^{\boxtimes s}) \to \Omega^s_{{-}/\Zbb},\\
	&h_0^\Nis(\Gbb_a^+\boxtimes \Gbb_a^+) \to \Pcal^1.
\end{align*}
Now we use the following result due to R\"ulling-Sugiyama-Yamazaki:

\begin{theorem}[{\cite[Theorem 5.6 and 5.14]{RSY}}]\label{RSY}
	The canonical symbol gives an isomorphism
	$$
		h_0^\Nis((\Gbb_m^+)^{\boxtimes s})\xrightarrow{\sim} K^M_s.
	$$
	If $\ch(k)\neq 2$, then the infinitesimal symbol gives an isomorphism
	$$
		h_0^\Nis(\Gbb_a^+\boxtimes \Gbb_a^+)\xrightarrow{\sim} \Pcal^1.
	$$
\end{theorem}

\begin{corollary}
	Let $\varphi$ be a bounded $(r,s)$-symbol for $F\in \RSC_\Nis$.
	Then $\varphi$ satisfies (ST1).
	If $\ch(k)\neq 2$, then it also satisfies (ST3).
\end{corollary}

\begin{proof}
	For a field $K$ and $a_1,\dots,a_r\in K,\,b_1,\dots,b_s\in K^\times$, the value $\varphi_K(a_1,\dots,a_r|b_1,\dots,b_s)$ is equal to the image of $(a_1,\dots,a_r)\otimes (b_1,\dots,b_s)$ under the morphism
	$$
		h_0^\Nis((\mathbb{G}_a^+)^{\boxtimes r})\otimes^\NST h_0^\Nis((\mathbb{G}_m^+)^{\boxtimes s})\to h_0^\Nis((\mathbb{G}_a^+)^{\boxtimes r}\boxtimes (\mathbb{G}_m^+)^{\boxtimes s}) \xrightarrow{\varphi} F
	$$
	evaluated at $K$.
	Now Theorem \ref{RSY} implies that we have
	$$
		(\dots,a,\dots,1-a,\dots)=0\in h_0^\Nis((\mathbb{G}_m^+)^{\boxtimes s})(K),
	$$
	and if $\ch(k)\neq 2$ then
	\begin{align*}
		&(\dots,bc,\dots,ad,\dots)+(\dots,ac,\dots,bd,\dots)\\
		 = &(\dots,c,\dots,abd,\dots)+(\dots,abc,\dots,d,\dots) \in h_0^\Nis((\mathbb{G}_a^+)^{\boxtimes r})(K).
	\end{align*}
	The claim follows from these relations.
\end{proof}

It remains to establish (ST2).
This will occupy the rest of the paper.

\section{Auxiliary symbols}

We associate two auxiliary symbols $\varphi^A,\,\varphi^B$ to each bounded $(r,s)$-symbol $\varphi$.
We show that (ST2) for $\varphi$ can be reduced to a certain condition for $\varphi^B$.

\begin{definition}
	Let $\varphi$ be a bounded $(r,s)$-symbol for $F\in \RSC_\Nis$ and suppose that $r\geq 1$.
	Let $\mu\colon \Gbb_a^+\boxtimes \Gbb_a^+\to \Gbb_a^+$ denote the morphism induced by the multiplication map
	$$
		\mathbb{A}^1\times\mathbb{A}^1\to \mathbb{A}^1;\quad(x,y)\mapsto xy.
	$$
	We define $\varphi^A$ to be the bounded $(r+1,s)$-symbol given by
	$$
		h_0^\Nis((\Gbb_a^+)^{\boxtimes (r+1)}\boxtimes (\Gbb_m^+)^{\boxtimes s})
		\xrightarrow{h_0^\Nis(\mu\boxtimes \id)} h_0^\Nis((\Gbb_a^+)^{\boxtimes r}\boxtimes (\Gbb_m^+)^{\boxtimes s})\xrightarrow{\varphi} F.
	$$
	Concretely, it can be written as
	$$
		\varphi^A_K(a_1,a_2,\dots|\dots)=\varphi_K(a_1a_2,\dots|\dots ).
	$$
\end{definition}

The auxiliary symbol $\varphi^A$ can be used to deduce the following

\begin{proposition}\label{prop:ST2_weak}
	Let $F\in \RSC_\Nis$.
	Suppose that any bounded $(r,s)$-symbol $\varphi$ for $F$ with $r,s\geq 1$ satisfies
		\begin{description}
			\item[{\bf (ST2')}]		$\varphi_K(\dots,a|a,\dots)+\varphi_K(\dots,1-a|1-a, \dots)=0$
		\end{description}
	for $K\in \Field_k,\,a\in K^\times\setminus\{1\}$.
	Then any bounded $(r,s)$-symbol $\varphi$ for $F$ satisfies (ST2).
\end{proposition}

\begin{proof}
	Since $\varphi^A$ (with variables permuted) satisfies the above relation, we have
	$$
		\varphi^A_K(c,a,\dots|a,\dots)+\varphi^A_K(c,1-a,\dots|1-a,\dots) = 0
	$$
	and this is nothing but (ST2) with variables permuted.
\end{proof}

\begin{definition}
	Let $\varphi$ be a bounded $(r,s)$-symbol for $F\in \RSC_\Nis$ and suppose that $r,s\geq 1$.
	Let $\delta\colon \mathbb{D}_2^+\to \Gbb_a^+\boxtimes \Gbb_m^+$ denote the morphism induced by the map
	$$
		\mathbb{A}^1\setminus\{0\}\to \mathbb{A}^1\times (\mathbb{A}^1\setminus \{0\});\quad x\mapsto (x,x).
	$$
	We define $\varphi^B$ to be the $(\underbrace{\Gbb_a^\#,\dots,\Gbb_a^\#}_{r-1},\mathbb{D}_2^\#,\underbrace{\Gbb_m^\#,\dots,\Gbb_m^\#}_{s-1})$-symbol given by
	$$
		h_0^\Nis((\Gbb_a^+)^{\boxtimes (r-1)}\boxtimes \mathbb{D}_2^+\boxtimes (\Gbb_m^+)^{\boxtimes (s-1)})
		\xrightarrow{h_0^\Nis(\id\boxtimes\delta\boxtimes \id)} h_0^\Nis((\Gbb_a^+)^{\boxtimes r}\boxtimes (\Gbb_m^+)^{\boxtimes s})\xrightarrow{\varphi} F.
	$$
	By definition, it satisfies
	$$
		\varphi^B_K(a_1,\dots,a_{r-1},[\![b]\!],c_1,\dots,c_{s-1}) = \varphi_K(a_1,\dots,a_{r-1},b|b,c_1,\dots,c_{s-1}).
	$$
\end{definition}

Now observe that there is a sequence of morphisms $\mathbb{D}_2^+\to \mathbb{W}_2^+ \to \mathbb{G}_a^+$ induced by the canonical inclusion $\mathbb{A}^1\setminus \{0\}\to \mathbb{A}^1$ and the identity map $\mathbb{A}^1\to \mathbb{A}^1$.
Taking $h_0^\Nis$ of this sequence, we recover the sequence of canonical epimorphisms
$$
	\mathbb{D}_2\xrightarrow{\pr_2} \mathbb{W}_2\twoheadrightarrow \mathbb{G}_a;\quad [\![a]\!]\mapsto [a]\mapsto a.
$$

\begin{lemma}\label{ch2_cor_mixed}
	Let $\varphi$ be a bounded $(r,s)$-symbol for $F\in \RSC_\Nis$ and suppose that $r,s\geq 1$.
	If $\varphi^B_K$ descends through the canonical surjection $\mathbb{D}_2(K)\xrightarrow{\pr_2} \mathbb{W}_2(K)\twoheadrightarrow K$ to a map from $K^{r-1}\times K\times (K^\times)^{s-1}$, then $\varphi_K$ satisfies (ST2').
\end{lemma}

\begin{proof}
	We regard $\varphi_K^B$ as a map $K^{r-1}\times K\times (K^\times)^{s-1} \to F(K)$.
	For any $a\in K^\times\setminus\{1\}$ we have
	\begin{align*}
			&	\varphi_K(\dots,a|a,\dots)+\varphi_K(\dots,1-a|1-a,\dots)\\
		=	&	\varphi^B_K(\dots,a,\dots)+\varphi^B_K(\dots,1-a,\dots)\\
		=	&	\varphi^B_K(\dots,1,\dots) = \varphi_K(\dots,1|1,\dots)=0.
	\end{align*}
	Here, for the second equality, we used the fact that $a+(1-a)=1$ holds in $K$.
	Hence $\varphi_K$ satisfies (ST2').
\end{proof}

\begin{lemma}\label{ch2_lemma_mixed}
	Assume $\ch (k)\neq 2$.
	Let $\varphi$ be a bounded $(r,s)$-symbol for $F\in \RSC_\Nis$ and suppose that $r,s\geq 1$.
	Suppose that $\varphi^B_K$ descends through the canonical surjection $\mathbb{D}_2(K)\xrightarrow{\pr_2} \mathbb{W}_2(K)$ to a map from $K^{r-1}\times \mathbb{W}_2(K)\times (K^\times)^{s-1}$ for all $K\in \Field_k$.
	Then it further descends through the canonical surjection $\mathbb{D}_2(K)\xrightarrow{\pr_2} \mathbb{W}_2(K)\twoheadrightarrow K$ to a map from $K^{r-1}\times K\times (K^\times)^{s-1}$ for all $K\in \Field_k$.
\end{lemma}

\begin{proof}
	We regard $\varphi_K^B$ as a map $K^{r-1}\times \mathbb{W}_2(K)\times (K^\times)^{s-1} \to F(K)$.
	We have to prove $\varphi^B_K(\dots,(0,a),\dots)=0$ for all $a\in K$.
	Since $2(0,a/2)=(0,a)$ holds in $\Wbb_2(K)$, it suffices to show that $2\varphi^B_K(\dots,(0,a),\dots)=0$.
	Set $L=K(\alpha)$ where $\alpha$ is a square root of $a$.
	Since $(0,a)=[\alpha]+[-\alpha]$ holds in $\Wbb_2(L)$, we have
	\begin{align*}
		\varphi^B_L(\dots,(0,a),\dots)&=\varphi^B_L(\dots,[\alpha],\dots)+\varphi^B_L(\dots,[-\alpha],\dots)\\
		&=\varphi_L(\dots,\alpha|\alpha,\dots) + \varphi_L(\dots,{-\alpha}|{-\alpha},\dots)\\
		&=\varphi_L(\dots,\alpha|\alpha,\dots) + \varphi_L(\dots,{-\alpha}|\alpha,\dots) + \varphi_L(\dots,{-\alpha}|{-1},\dots)\\
		&=\varphi_L(\dots,{-\alpha}|{-1},\dots)=\varphi_L(\dots,{-\alpha}/2,|1,\dots)=0.
	\end{align*}
	Applying $\Tr_{L/K}$ we get $2\varphi^B_K(\dots,(0,a),\dots)=0$.
\end{proof}

Therefore, to complete the proof of Theorem \ref{main}, it suffices to show that the morphism
	$$
		h_0^\Nis((\Gbb_a^+)^{\boxtimes (r-1)}\boxtimes \mathbb{D}_2^+\boxtimes (\Gbb_m^+)^{\boxtimes (s-1)})
		\xrightarrow{h_0^\Nis(\id\boxtimes\delta\boxtimes \id)} h_0^\Nis((\Gbb_a^+)^{\boxtimes r}\boxtimes (\Gbb_m^+)^{\boxtimes s})
	$$
factors through $h_0^\Nis((\Gbb_a^+)^{\boxtimes (r-1)}\boxtimes \mathbb{W}_2^+\boxtimes (\Gbb_m^+)^{\boxtimes (s-1)})$.
We will prove the following generalization of this:

\begin{theorem}\label{main3}
	For any proper modulus fraction $\Xcal/\Ycal$ and $n\geq 1$, the morphism
	$$
		h_0^\Nis(\mathbb{D}_{n+1}^+\boxtimes \Xcal/\Ycal)
		\xrightarrow{h_0^\Nis(\delta\boxtimes \id)} h_0^\Nis(\Wbb_n^+\boxtimes \Gbb_m^+\boxtimes \Xcal/\Ycal)
	$$
	factors through $h_0^\Nis(\mathbb{W}_{n+1}^+\boxtimes \Xcal/\Ycal)$.
\end{theorem}

\section{Completion of the proof}

In this section we prove Theorem \ref{main3}.

\begin{lemma}\label{CIrec_factorization_lemma}
	Let $g\colon \Xcal/\Ycal\to \Xcal'/\Ycal'$ be a morphism of proper modulus fractions.
	Suppose that for any $G\in \CI^{\tau,sp}_\Nis$, the homomorphism $G(\Xcal'/\Ycal')\to G(\Xcal/\Ycal)$ induced by $g$ is surjective.
	Then the morphism
	$h_0^\Nis\biggl(\dfrac{\mathcal{X}}{\mathcal{Y}}\boxtimes \dfrac{\mathcal{Z}}{\mathcal{W}}\biggr) \to 
		h_0^\Nis\biggl(\dfrac{\mathcal{X}'}{\mathcal{Y}'}\boxtimes \dfrac{\mathcal{Z}}{\mathcal{W}}\biggr)$
	induced by $g$ is a split monomorphism for any proper modulus fraction $\mathcal{Z}/\mathcal{W}$.
\end{lemma}

\begin{proof}
	Set $F:=h_0^\Nis\biggl(\dfrac{\mathcal{X}}{\mathcal{Y}}\boxtimes \dfrac{\mathcal{Z}}{\mathcal{W}}\biggr)$; this is a reciprocity sheaf and we have a unit morphism $h_0^\cube\biggl(\dfrac{\mathcal{X}}{\mathcal{Y}}\boxtimes \dfrac{\mathcal{Z}}{\mathcal{W}}\biggr)\to \tilde{F}$.
	Since $(\tau_!^\Nis\tilde{F})^{\mathcal{Z}/\mathcal{W}}$ is an object of $\CIrec$ by Lemma \ref{tilde_in_CIrec}, our assumption implies that the induced morphism
	$\tau_!^\Nis h_0^\cube\biggl(\dfrac{\mathcal{X}}{\mathcal{Y}}\boxtimes \dfrac{\mathcal{Z}}{\mathcal{W}}\biggr)\to \tau_!^\Nis \tilde{F}$
	factors through $\tau_!^\Nis h_0^\cube\biggl(\dfrac{\mathcal{X'}}{\mathcal{Y'}}\boxtimes \dfrac{\mathcal{Z}}{\mathcal{W}}\biggr)$.
	Applying $\underline{\omega}_!$ and using $\underline{\omega}_!\tau_!^\Nis\simeq \omega_!^\Nis$, we get the desired retraction.
\end{proof}

\begin{lemma}\label{toric_contractibility}
	Consider the surface $\Bl_{(0:0:1)}\Pbb^2$.
	Let $L_0$ be the strict transform of $[\ast:0:\ast]\simeq \Pbb^1$ and $L_\infty = [\ast:\ast:0]\simeq \Pbb^1$.
	Then the canonical homomorphism
	$$
		H^i_\Nis(\Spec k, G)\to H^i_\Nis((\Bl_{(0:0:1)}\Pbb^2,L_0+L_\infty),G)
	$$
	is an isomorphism for any $G\in \CIrec$ and $i\geq 0$.
\end{lemma}

\begin{proof}
	Identify $\Bl_{(0:0:1)}\Pbb^2$ with the closed subscheme $\{((x:y:z),(s:t))\in \Pbb^2\times \Pbb^1\mid xt=ys\}$ of $\Pbb^2\times \Pbb^1$.
	Then the composition $\pi\colon \Bl_{(0:0:1)}\Pbb^2\to \Pbb^2\times \Pbb^1\xrightarrow{\pr_2} \Pbb^1$ exhibits $\Bl_{(0:0:1)}\Pbb^2$ as a $\Pbb^1$-bundle over $\Pbb^1$.
	Moreover, $L_0$ corresponds to $\pi^{-1}(1:0)$ and $L_\infty$ corresponds to a section of this $\Pbb^1$-bundle.
	We see that the ambient morphism $\pi\colon (\Bl_{(0:0:1)}\Pbb^2, L_0+L_\infty)\to (\Pbb^1,\{(1:0)\})$ is a cube-bundle, and hence the claim follows from Lemma \ref{cube_invariance}.
\end{proof}

\begin{definition}
	Let $(\Xcal,\Ycal,f)$ be a modulus fraction and $G\in \CIrec$.
	We define $G_{\Xcal/\Ycal}$ to be the fiber of the canonical morphism $G_{\Xcal}\to \mathrm{R}f_*G_{\Ycal}$ in $D(\Sh_\Nis(X))$ and $H^i_\Nis\left({\Xcal}/{\Ycal},G\right) := R^i\Gamma(X,G_{\Xcal/\Ycal})$.
\end{definition}

	By definition, we have a long exact sequence
	$$
	\cdots\to H^i_\Nis(\Xcal/\Ycal,G)\to H^i_\Nis(\Xcal,G)\to H^i_\Nis(\Ycal,G)\to \cdots.
	$$
	In particular we have $H^0_\Nis(\Xcal/\Ycal,G) = G(\Xcal/\Ycal)$.
	We also have a Mayer-Vietoris sequence:

\begin{lemma}\label{excision_lemma}
	Let $f\colon \Xcal/\Ycal\to \Xcal'/\Ycal'$ be a morphism of modulus fractions and $G\in \CIrec$.
	Suppose that there is an elementary Nisnevich square
	$$
		\xymatrix{
			W\ar[r]\ar[d]	&V\ar[d]^-{\text{\'etale}}\\
			U\ar[r]^-{\text{open}}			&X.
		}
	$$
	Then the sequence
	$$
	\cdots\to H^i_\Nis(\Xcal/\Ycal,G)\to H^i_\Nis(\Xcal_U/\Ycal_U,G)\oplus H^i_\Nis(\Xcal_V/\Ycal_V,G) \xrightarrow{-} H^i_\Nis(\Xcal_W/\Ycal_W,G)\to \cdots
	$$
	is exact, where $({-})_S$ denotes the base change to $S$.
\end{lemma}

\begin{proof}
	The commutative square
	$$
		\xymatrix{
			\RGamma(X,G_{\Xcal})\ar[r]\ar[d] &\RGamma(V,G_{\Xcal_V})\ar[d]\\
			\RGamma(U,G_{\Xcal_U})\ar[r]	 &\RGamma(W,G_{\Xcal_W})
		}
	$$
	is homotopy Cartesian since $G_{\Xcal}$ is a Nisnevich sheaf.
	We also have a similar homotopy Cartesian diagram for $\Ycal$.
	Taking fibers, we see that
	$$
		\xymatrix{
			\RGamma(X,G_{\Xcal/\Ycal})\ar[r]\ar[d] &\RGamma(V,G_{\Xcal_V/\Ycal_V})\ar[d]\\
			\RGamma(U,G_{\Xcal_U/\Ycal_U})\ar[r]	 &\RGamma(W,G_{\Xcal_W/\Ycal_W})
		}
	$$
	is homotopy Cartesian and hence the result.
\end{proof}

A {\it standard point} on $\mathbb{P}^1\times \mathbb{P}^1$ is a $k$-rational point of the form $(\sigma,\tau)$ where $\sigma,\tau\in \{0,1,\infty\}$.
A {\it standard surface} is a $k$-scheme obtained by blowing up $\mathbb{P}^1\times \mathbb{P}^1$ along standard points.
When we consider a standard surface, we use the following notation:
\begin{align*}
		&E_{\sigma\tau}:=\text{exceptional divisor above $(\sigma,\tau)$}~(\sigma,\tau\in \{0,1,\infty\}),\\
		&D_\sigma:=\text{(strict transform of $\Pbb^1\times\{\sigma\}$)},\quad D'_\sigma:=\text{(strict transform of $\{\sigma\}\times\Pbb^1$)}~(\sigma\in \{0,1,\infty\}).
\end{align*}

\begin{lemma}\label{key_lemma}
	Let $G\in \CIrec$ and $B=\Bl_{(0,0)}\Abb^2$.
	Let $D_x$ (resp. $D_y$) denote the strict transform of the $x$-axis (resp. $y$-axis) and $E$ the exceptional divisor.
	Then
	$$
		H^i_\Nis\biggl(\frac{(B,D_x)}{D_y},G\biggr)\to H^i_\Nis\biggl(\frac{(B,D_x+E)}{(D_y,E|_{D_y})},G\biggr)
	$$
	is an isomorphism for each $i\geq 0$.
\end{lemma}

\begin{proof}
	Consider a standard surface $B':=\Bl_{\{(0,0),(\infty,\infty)\}}(\Pbb^1\times \Pbb^1)$.
	We can blow-down $B'$ along $D_\infty$ and $D'_\infty$ to get $\Bl_{(0:0:1)}\Pbb^2$.
	$$
		\begin{tikzpicture}
	\begin{pgfonlayer}{nodelayer}
		\node [style=none] (0) at (4, -3) {};
		\node [style=none] (1) at (6, -2) {};
		\node [style=none] (3) at (3, -2) {};
		\node [style=none] (4) at (4, 0) {};
		\node [style=none] (5) at (3, -1) {};
		\node [style=none] (6) at (5, -3) {};
		\node [style=none] (7) at (4, -3) {};
		\node [style=none] (12) at (3, 3) {};
		\node [style=none] (13) at (9, -3) {};
		\node [style=none] (14) at (10, -2) {};
		\node [style=none] (15) at (4, 4) {};
		\node [style=none] (16) at (4, 4.5) {};
		\node [style=none] (17) at (4, 4.5) {$y$};
		\node [style=none] (19) at (8.5, 2) {$L_\infty$};
		\node [style=none] (24) at (6.5, -1.5) {$L_0$};
		\node [style=none] (25) at (10.5, -2) {$x$};
		\node [style=none] (26) at (-9, -3) {};
		\node [style=none] (27) at (-7, -2) {};
		\node [style=none] (28) at (-10, -2) {};
		\node [style=none] (29) at (-9, 0) {};
		\node [style=none] (30) at (-10, -1) {};
		\node [style=none] (31) at (-8, -3) {};
		\node [style=none] (32) at (-9, -3) {};
		\node [style=none] (33) at (-7, 2) {};
		\node [style=none] (34) at (-5, 3) {};
		\node [style=none] (35) at (-4, 2) {};
		\node [style=none] (36) at (-5, 0) {};
		\node [style=none] (37) at (-6, 3) {};
		\node [style=none] (38) at (-4, 1) {};
		\node [style=none] (39) at (-3, -2) {};
		\node [style=none] (40) at (-9, 4) {};
		\node [style=none] (41) at (-9, 4.5) {};
		\node [style=none] (42) at (-9, 4.5) {$y$};
		\node [style=none] (43) at (-5.75, 1.25) {};
		\node [style=none] (44) at (-3.5, 0.5) {$E_{\infty \infty}$};
		\node [style=none] (45) at (-7, 2.5) {$D_{\infty}$};
		\node [style=none] (46) at (-5.75, 0) {$D'_{\infty}$};
		\node [style=none] (47) at (-8, -3.5) {$E_{00}$};
		\node [style=none] (48) at (-9.75, 0) {$D'_{0}$};
		\node [style=none] (49) at (-7, -1.5) {$D_{0}$};
		\node [style=none] (50) at (-2.5, -2) {$x$};
		\node [style=none] (51) at (-10, 2) {};
		\node [style=none] (52) at (-5, -3) {};
		\node [style=none] (53) at (-1, 0) {};
		\node [style=none] (54) at (1, 0) {};
		\node [style=none] (55) at (-7, -5) {$B'=\mathrm{Bl}_{\{(0,0),(\infty,\infty)\}}(\mathbb{P}^1\times \mathbb{P}^1)$};
		\node [style=none] (56) at (6.5, -5) {$\mathrm{Bl}_{(0:0:1)}\mathbb{P}^2$};
	\end{pgfonlayer}
	\begin{pgfonlayer}{edgelayer}
		\draw [in=90, out=180, looseness=1.25] (1.center) to (0.center);
		\draw [in=0, out=-90, looseness=1.25] (4.center) to (3.center);
		\draw (5.center) to (6.center);
		\draw [bend left=45, looseness=1.25] (12.center) to (13.center);
		\draw [style=new edge style 0] (1.center) to (14.center);
		\draw [style=new edge style 0] (4.center) to (15.center);
		\draw [in=90, out=180, looseness=1.25] (27.center) to (26.center);
		\draw [in=0, out=-90, looseness=1.25] (29.center) to (28.center);
		\draw (30.center) to (31.center);
		\draw [in=0, out=-90, looseness=1.25] (34.center) to (33.center);
		\draw [in=90, out=-180, looseness=1.25] (35.center) to (36.center);
		\draw (37.center) to (38.center);
		\draw [style=new edge style 0] (27.center) to (39.center);
		\draw [style=new edge style 0] (29.center) to (40.center);
		\draw (33.center) to (51.center);
		\draw (36.center) to (52.center);
		\draw [style=new edge style 0] (53.center) to (54.center);
	\end{pgfonlayer}
\end{tikzpicture}

	$$
	By Theorem \ref{log_blowup}, we have the following commutative diagram:
	$$
		\xymatrix{
		&H^i_\Nis((B',D_0+D_\infty+D'_\infty+E_{\infty\infty}),G)\ar[d]\\
		H^i_\Nis((\Bl_{(0:0:1)}\Pbb^2,L_0+L_\infty),G) \ar[r]\ar[rd]_-\alpha\ar[ur]^-\sim		&H^i_\Nis((B',D_0+D_\infty+2D'_\infty+E_{\infty\infty}),G)\ar[d]^-\beta\\
																&H^i_\Nis(\Spec k,G).
		}
	$$
	Here $\alpha$ (resp. $\beta$) is induced by $(1:1:1)\colon \Spec k\to \Bl_{(0:0:1)}\Pbb^2$ (resp. $(1,1)\colon \Spec k\to B'$).
	By Lemma \ref{toric_contractibility}, $\alpha$ is an isomorphism.
	It follows that the composition of the right vertical arrows is an isomorphism, and hence
	$$
		H^i_\Nis(\Spec k,G)\xrightarrow{\sim} H^i_\Nis((B',D_0+D_\infty+D'_\infty+E_{\infty\infty}),G).
	$$
	Since $H^i_\Nis(\Spec k,G)\to H^i_\Nis((D'_0,D_\infty|_{D'_0}),G)$ is also an isomorphism by Theorem \ref{cube_invariance}, we get
	\begin{align}
		\label{vanishing1} H^i_\Nis\biggl(\dfrac{(B',D_0+D_\infty+D'_\infty+E_{\infty\infty})}{(D'_0,D_\infty|_{D'_0})},G\biggr)=0\quad(i\geq 0)
	\end{align}
	by the long exact sequence.
	
	On the other hand, we can blow-down $B'$ along $E_{00}$ and $E_{\infty\infty}$ to get $\mathbb{P}^1\times \mathbb{P}^1$.
	$$
		\begin{tikzpicture}
	\begin{pgfonlayer}{nodelayer}
		\node [style=none] (26) at (-9, -3) {};
		\node [style=none] (27) at (-7, -2) {};
		\node [style=none] (28) at (-10, -2) {};
		\node [style=none] (29) at (-9, 0) {};
		\node [style=none] (30) at (-10, -1) {};
		\node [style=none] (31) at (-8, -3) {};
		\node [style=none] (32) at (-9, -3) {};
		\node [style=none] (33) at (-7, 2) {};
		\node [style=none] (34) at (-5, 3) {};
		\node [style=none] (35) at (-4, 2) {};
		\node [style=none] (36) at (-5, 0) {};
		\node [style=none] (37) at (-6, 3) {};
		\node [style=none] (38) at (-4, 1) {};
		\node [style=none] (39) at (-3, -2) {};
		\node [style=none] (40) at (-9, 4) {};
		\node [style=none] (41) at (-9, 4.5) {};
		\node [style=none] (42) at (-9, 4.5) {$y$};
		\node [style=none] (44) at (-3.5, 0.5) {$E_{\infty \infty}$};
		\node [style=none] (45) at (-7, 2.5) {$D_{\infty}$};
		\node [style=none] (46) at (-5.75, 0) {$D'_{\infty}$};
		\node [style=none] (47) at (-8, -3.5) {$E_{00}$};
		\node [style=none] (48) at (-9.75, 0) {$D'_{0}$};
		\node [style=none] (49) at (-7, -1.5) {$D_{0}$};
		\node [style=none] (50) at (-2.5, -2) {$x$};
		\node [style=none] (51) at (-10, 2) {};
		\node [style=none] (52) at (-5, -3) {};
		\node [style=none] (53) at (-1, 0) {};
		\node [style=none] (54) at (1, 0) {};
		\node [style=none] (55) at (-7, -5) {$B'=\mathrm{Bl}_{\{(0,0),(\infty,\infty)\}}(\mathbb{P}^1\times \mathbb{P}^1)$};
		\node [style=none] (56) at (7.5, -5) {$\mathbb{P}^1\times\mathbb{P}^1$};
		\node [style=none] (57) at (4, -2) {};
		\node [style=none] (58) at (7, -2) {};
		\node [style=none] (59) at (5, -3) {};
		\node [style=none] (60) at (5, 0) {};
		\node [style=none] (64) at (7, 2) {};
		\node [style=none] (65) at (10, 2) {};
		\node [style=none] (66) at (9, 3) {};
		\node [style=none] (67) at (9, 0) {};
		\node [style=none] (70) at (11, -2) {};
		\node [style=none] (71) at (5, 4) {};
		\node [style=none] (72) at (5, 4.5) {};
		\node [style=none] (73) at (5, 4.5) {$y$};
		\node [style=none] (76) at (7, 2.5) {$\mathbb{P}^1\times\{\infty\}$};
		\node [style=none] (78) at (3.5, 0) {$\{0\}\times \mathbb{P}^1$};
		\node [style=none] (79) at (7, -1.5) {$\mathbb{P}^1\times \{0\}$};
		\node [style=none] (80) at (11.5, -2) {$x$};
		\node [style=none] (81) at (4, 2) {};
		\node [style=none] (82) at (9, -3) {};
		\node [style=none] (83) at (10.75, 0) {$\{\infty\}\times \mathbb{P}^1$};
	\end{pgfonlayer}
	\begin{pgfonlayer}{edgelayer}
		\draw [in=90, out=180, looseness=1.25] (27.center) to (26.center);
		\draw [in=0, out=-90, looseness=1.25] (29.center) to (28.center);
		\draw (30.center) to (31.center);
		\draw [in=0, out=-90, looseness=1.25] (34.center) to (33.center);
		\draw [in=90, out=-180, looseness=1.25] (35.center) to (36.center);
		\draw (37.center) to (38.center);
		\draw [style=new edge style 0] (27.center) to (39.center);
		\draw [style=new edge style 0] (29.center) to (40.center);
		\draw (33.center) to (51.center);
		\draw (36.center) to (52.center);
		\draw [style=new edge style 0] (53.center) to (54.center);
		\draw (58.center) to (57.center);
		\draw (60.center) to (59.center);
		\draw (65.center) to (64.center);
		\draw (66.center) to (67.center);
		\draw [style=new edge style 0] (58.center) to (70.center);
		\draw [style=new edge style 0] (60.center) to (71.center);
		\draw (64.center) to (81.center);
		\draw (67.center) to (82.center);
	\end{pgfonlayer}
\end{tikzpicture}

	$$
	By Theorem \ref{log_blowup}, we have the following commutative diagram:
	$$
		\xymatrix{
		&H^i_\Nis((B',D_0+E_{00}+D_\infty+D'_\infty+E_{\infty\infty}),G)\ar[d]\\
		H^i_\Nis((\Pbb^1\times\Pbb^1,\{\infty\}\times\Pbb^1+\Pbb^1\times \{0,\infty\}),G) \ar[r]\ar[d] \ar[ur]^-\sim
		&H^i_\Nis((B',D_0+E_{00}+D_\infty+D'_\infty+2E_{\infty\infty}),G)\ar[d]\\
		H^i_\Nis((\{0\}\times \Pbb^1,\{0\}\times \{0,\infty\}),G)\ar[r]^-\sim	&H^i_\Nis((D'_0,D_\infty|_{D'_0}+E_{00}|_{D'_0}),G).
		}
	$$
	Theorem \ref{cube_invariance} implies that the left vertical arrow is an isomorphism, and hence so is the right vertical composite.
	Therefore we get
	\begin{align}
		\label{vanishing2}H^i_\Nis\biggl(\dfrac{(B',D_0+E_{00}+D_\infty+D'_\infty+E_{\infty\infty})}{(D'_0,D_\infty|_{D'_0}+E_{00}|_{D'_0})},G\biggr)=0\quad(i\geq 0)
	\end{align}
	by the long exact sequence.
	
	Now consider the Zariski open covering $B'=B\cup V$ where $V=(\Pbb^1\times\Pbb^1)\setminus\{(0,0)\}$.
	Applying Lemma \ref{excision_lemma} to this open covering and using (\ref{vanishing1}) and (\ref{vanishing2}), we get the desired result.
\end{proof}

\begin{proposition}\label{key_lemma_general}
	Let $\Xcal/\Ycal$ be a modulus fraction such that $X$ is a smooth surface.
	Let $x$ be a $k$-rational point of $X$.
	Suppose that there is a regular sysmtem of parameters $t_1,t_2$ at $x$ such that $D_X=\{t_1=0\}$, $Y=\{t_2=0\}$ and $D_Y=\{x\}$ locally around $x$.
	Then
	$$
		H^i_\Nis\biggl(\dfrac{(\Bl_xX,\tilde{D}_X)}{(Y,D_Y\setminus\{x\})},G\biggr)\to H^i_\Nis\biggl(\dfrac{(\Bl_xX,\tilde{D}_X+E)}{(Y,D_Y)},G\biggr)
	$$
	is an isomorphism for any $G\in \CIrec$ and $i\geq 0$, where $\tilde{D}_X$ is the strict transform of $D_X$ and $E$ is the exceptional divisor.
	If $\mathcal{X}/\mathcal{Y}$ is proper, then the morphism
	$$
		h_0^\Nis\biggl(\dfrac{(\Bl_xX,\tilde{D}_X+E)}{(Y,D_Y)}\boxtimes \dfrac{\mathcal{Z}}{\mathcal{W}}\biggr)\to h_0^\Nis\biggl(\dfrac{(\Bl_xX,\tilde{D}_X)}{(Y,D_Y\setminus\{x\})}\boxtimes \dfrac{\mathcal{Z}}{\mathcal{W}}\biggr)
	$$
	is a split monomorphism for any proper modulus fraction $\mathcal{Z}/\mathcal{W}$.
\end{proposition}

\begin{proof}
	The second assertion follows from the first one and Lemma \ref{CIrec_factorization_lemma}.
	Suppose that $\pi\colon V\to X$ is an \'etale morphism with $V\in \Sm$ such that $\pi^{-1}(x)\to \{x\}$ is an isomorphism.
	Setting $U=X\setminus \{x\}$ and $\pi^{-1}(x)=\{x'\}$, we get a Nisnevich distinguished square
	$$
	\xymatrix{
		\pi^{-1}(U)\ar[r]\ar[d]		&\Bl_{x'}V\ar[d]^-{\text{\'etale}}\\
		U\ar[r]^-{\text{open}}					&\Bl_xX.
	}
	$$
	By Lemma \ref{excision_lemma}, we see that the claim for $(\Xcal/\Ycal,x)$ is equivalent to that for $(\Xcal_V/\Ycal_V,{x'})$.
	Therefore we may reduce to the case $X=\Abb^2$, $x=(0,0)$, $D_X=(x\text{-axis})$, $Y=(y\text{-axis})$ and $D_Y=\{(0,0)\}$.
	In this case the claim follows from Lemma \ref{key_lemma}.
\end{proof}

\begin{proof}[\bf Proof of Theorem \ref{main3}]
	Consider a standard surface $B:=\Bl_{(0,0)}(\Pbb^1\times \Pbb^1)$.
	$$
		\begin{tikzpicture}
	\begin{pgfonlayer}{nodelayer}
		\node [style=none] (26) at (-2, -3) {};
		\node [style=none] (27) at (0, -2) {};
		\node [style=none] (28) at (-3, -2) {};
		\node [style=none] (29) at (-2, 0) {};
		\node [style=none] (30) at (-3, -1) {};
		\node [style=none] (31) at (-1, -3) {};
		\node [style=none] (32) at (-2, -3) {};
		\node [style=none] (33) at (0, 2) {};
		\node [style=none] (34) at (3, 2) {};
		\node [style=none] (35) at (2, 3) {};
		\node [style=none] (36) at (2, 0) {};
		\node [style=none] (39) at (4, -2) {};
		\node [style=none] (40) at (-2, 4) {};
		\node [style=none] (41) at (-2, 4.5) {};
		\node [style=none] (42) at (-2, 4.5) {$y$};
		\node [style=none] (45) at (0, 2.5) {$D_{\infty}$};
		\node [style=none] (46) at (1.25, 0) {$D'_{\infty}$};
		\node [style=none] (47) at (-1, -3.5) {$E_{00}$};
		\node [style=none] (48) at (-2.75, 0) {$D'_{0}$};
		\node [style=none] (49) at (0, -1.5) {$D_{0}$};
		\node [style=none] (50) at (4.5, -2) {$x$};
		\node [style=none] (51) at (-3, 2) {};
		\node [style=none] (52) at (2, -3) {};
		\node [style=none] (55) at (0, -5) {$B=\mathrm{Bl}_{(0,0)}(\mathbb{P}^1\times \mathbb{P}^1)$};
	\end{pgfonlayer}
	\begin{pgfonlayer}{edgelayer}
		\draw [in=90, out=180, looseness=1.25] (27.center) to (26.center);
		\draw [in=0, out=-90, looseness=1.25] (29.center) to (28.center);
		\draw (30.center) to (31.center);
		\draw [in=0, out=-180, looseness=0.75] (34.center) to (33.center);
		\draw [in=90, out=-90, looseness=0.75] (35.center) to (36.center);
		\draw [style=new edge style 0] (27.center) to (39.center);
		\draw [style=new edge style 0] (29.center) to (40.center);
		\draw (33.center) to (51.center);
		\draw (36.center) to (52.center);
	\end{pgfonlayer}
\end{tikzpicture}

	$$
	We define modulus pairs $\Bcal,\Dcal,\Bcal',\Dcal',\Ecal$ by
	\begin{align*}
		\Bcal &= (B, D_0+D_\infty+(n+1)D'_\infty+E_{00}),\\
		\Dcal &= (D'_0,D_\infty|_{D'_0}+E_{00}|_{D'_0}),\\
		\Bcal' &= (B, D_0+D_\infty+(n+1)D'_\infty),\\
		\Dcal' &= (D'_0,D_\infty|_{D'_0}),\\
		\Ecal &= (E_{00}, D_0|_{E_{00}}) \simeq \cube.
	\end{align*}
	There are canoical morphisms of modulus fractions
	$$
		\Bcal'/(\Dcal'\sqcup \Ecal) \leftarrow \Bcal/\Dcal \rightarrow \mathbb{W}_n^+\boxtimes \mathbb{G}_m^+.
	$$
	The induced morphism $h_0^\Nis\biggl(\dfrac{\Bcal}{\Dcal}\boxtimes\dfrac{\mathcal{X}}{\mathcal{Y}}\biggr)
	\to  h_0^\Nis\biggl(\dfrac{\Bcal'}{\Dcal'\sqcup \Ecal}\boxtimes\dfrac{\mathcal{X}}{\mathcal{Y}}\biggr)
	\simeq h_0^\Nis\biggl(\dfrac{\Bcal'}{\Dcal'}\boxtimes\dfrac{\mathcal{X}}{\mathcal{Y}}\biggr)$ is a split monomorphism for any proper modulus fraction $\mathcal{X}/\mathcal{Y}$ by Proposition \ref{key_lemma_general}.
	We define $\delta\colon \mathbb{P}^1\to B$ to be the strict transform of
	$\mathbb{P}^1\to \mathbb{P}^1\times \mathbb{P}^1;\; x\mapsto (x,x)$.
	$$
		\begin{tikzpicture}
	\begin{pgfonlayer}{nodelayer}
		\node [style=none] (26) at (-2, -3) {};
		\node [style=none] (27) at (0, -2) {};
		\node [style=none] (28) at (-3, -2) {};
		\node [style=none] (29) at (-2, 0) {};
		\node [style=none] (30) at (-3, -1) {};
		\node [style=none] (31) at (-1, -3) {};
		\node [style=none] (32) at (-2, -3) {};
		\node [style=none] (33) at (0, 2) {};
		\node [style=none] (34) at (3, 2) {};
		\node [style=none] (35) at (2, 3) {};
		\node [style=none] (36) at (2, 0) {};
		\node [style=none] (39) at (4, -2) {};
		\node [style=none] (40) at (-2, 4) {};
		\node [style=none] (41) at (-2, 4.5) {};
		\node [style=none] (42) at (-2, 4.5) {$y$};
		\node [style=none] (45) at (0, 2.5) {$D_{\infty}$};
		\node [style=none] (46) at (1.25, 0) {$D'_{\infty}$};
		\node [style=none] (47) at (-1, -3.5) {$E_{00}$};
		\node [style=none] (48) at (-2.75, 0) {$D'_{0}$};
		\node [style=none] (49) at (0, -1.5) {$D_{0}$};
		\node [style=none] (50) at (4.5, -2) {$x$};
		\node [style=none] (51) at (-3, 2) {};
		\node [style=none] (52) at (2, -3) {};
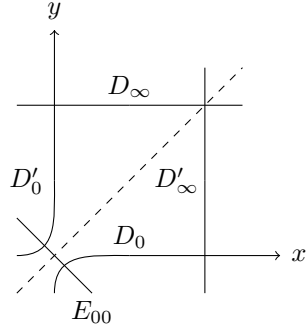
		\node [style=none] (55) at (0, -5) {The image of $\delta$};
		\node [style=none] (56) at (-3, -3) {};
		\node [style=none] (57) at (3, 3) {};
		\node [style=none] (58) at (-0.25, -1.25) {};
	\end{pgfonlayer}
	\begin{pgfonlayer}{edgelayer}
		\draw [in=90, out=180, looseness=1.25] (27.center) to (26.center);
		\draw [in=0, out=-90, looseness=1.25] (29.center) to (28.center);
		\draw (30.center) to (31.center);
		\draw [in=0, out=-180, looseness=0.75] (34.center) to (33.center);
		\draw [in=90, out=-90, looseness=0.75] (35.center) to (36.center);
		\draw [style=new edge style 0] (27.center) to (39.center);
		\draw [style=new edge style 0] (29.center) to (40.center);
		\draw (33.center) to (51.center);
		\draw (36.center) to (52.center);
		\draw [style=dashedline] (56.center) to (57.center);
	\end{pgfonlayer}
\end{tikzpicture}

	$$
	The morphism $\delta$ induces the following commutative diagram of modulus fractions:
	$$
	\xymatrix{
		\mathbb{D}_{n+1}^+\ar[r]^-{\delta}\ar[d]		&\Bcal/\Dcal\ar[r]\ar[d]			&\Wbb_n^+\boxtimes \Gbb_m^+\\
		\mathbb{W}_{n+1}^+\ar[r]^-{\delta}						&\Bcal'/(\Dcal'\sqcup \Ecal).
	}
	$$
	Applying $h_0^\Nis({-}\boxtimes \mathcal{X}/\mathcal{Y})$, we conclude that
	$$
		h_0^\Nis(\delta\boxtimes \id)\colon h_0^\Nis(\Dbb_{n+1}^+\boxtimes \Xcal/\Ycal)\to h_0^\Nis(\Wbb_n^+\boxtimes \Gbb_m^+\boxtimes \Xcal/\Ycal)
	$$
	factors through $h_0^\Nis(\Wbb_{n+1}^+\boxtimes \Xcal/\Ycal)$.
\end{proof}

Finally, we deduce the following result from our main theorem.
This generalizes \cite[Theorem 5.19]{RSY} which required $\ch(k)\not \in \{2,3,5\}$.

\begin{theorem}\label{application}
	If $\ch(k)\neq 2$, then the extended differential symbol gives an isomorphism
	$$
		h_0^\Nis((\Gbb_a^+)\boxtimes(\Gbb_m^+)^{\boxtimes s}) \xrightarrow{\sim} \Omega^s_{{-}/\Zbb}.
	$$
\end{theorem}

\begin{proof}
	Let $F\in \RSC_\Nis$.
	A bounded $(1,s)$-symbol $h_0^\Nis(\Gbb_a^+\boxtimes (\Gbb_m^+)^{\boxtimes s})\to F$ satisfies (ST1) and (ST2) if and only if it factors (uniquely) through the differential symbol $h_0^\Nis(\Gbb_a^+\boxtimes (\Gbb_m^+)^{\boxtimes s}) \twoheadrightarrow \Omega^s_{{-}/\Zbb}$; this follows from Lemma \ref{injectivity_factorization_lemma} and Lemma \ref{omega_universal}.
	However, Theorem \ref{main} says that {\it any} bounded $(1,s)$-symbol satifies (ST1) and (ST2), so any morphism $h_0^\Nis(\Gbb_a^+\boxtimes (\Gbb_m^+)^{\boxtimes s})\to F$ factors (uniquely) through $\Omega^s_{{-}/\Zbb}$.
	Since $h_0^\Nis((\Gbb_a^+)\boxtimes(\Gbb_m^+)^{\boxtimes s})$ itself is a reciprocity sheaf, we get the claimed isomorphism by the universal property.
\end{proof}

\vspace{.2in}

\printbibliography

\end{document}